\documentclass[11pt,A4]{article}
\usepackage[utf8]{inputenc}
\usepackage[english]{babel}
\usepackage{amsmath,amsfonts,amssymb,amsthm,mathrsfs,bbm}
\usepackage{mathtools}	% a besoin de amsmath
\mathtoolsset{showonlyrefs}	% ne numérote que les équations référencées
\usepackage[font=sf, labelfont={sf,bf}, margin=1cm]{caption}
\usepackage{graphicx,graphics}
\usepackage{epsfig}
\usepackage{latexsym}
 \usepackage{ae,aecompl}
\usepackage{pstricks}
\usepackage{enumerate}
\usepackage{xcolor}
\usepackage{pifont}
\usepackage[pdfpagemode=UseNone,bookmarksopen=false,colorlinks=true,urlcolor=blue,citecolor=blue,citebordercolor=blue,linkcolor=blue]{hyperref}
\usepackage[margin=2.5cm]{geometry}
\usepackage{comment}
\usepackage{mathpazo} 
\usepackage{eulervm}
\DeclareGraphicsExtensions{.jpg,.pdf}
\usepackage[normalem]{ulem}
\usepackage{enumerate,stmaryrd}
\usepackage{tikz}						%%%%%%%%%%
	\usetikzlibrary{arrows, patterns, calc}		%%%%%%%%%%
\usepackage{mathtools}					%%%%%%%%%%
\usepackage{microtype,etex}					%%%%%%%%%%

\usepackage{enumitem}

\newtheorem{thm}{Theorem}
\newtheorem{lem}[thm]{Lemma}
\newtheorem{prop}[thm]{Proposition}
\newtheorem{cor}[thm]{Corollary}

\theoremstyle{definition}

\newtheorem{defn}[thm]{Definition}
\newtheorem{ex}[thm]{Example}

\newtheorem{rem}[thm]{Remark}

\newtheorem{ques}[thm]{Question}

\renewcommand\Pr[1]{\mathbb{P}\left(#1\right)}
\newcommand\Es[1]{\mathbb{E}\left[#1\right]}

\def \P {\mathbb{P}}

\def \N {\mathbb N}

\def \D {\mathbb D}
\def \NC {\mathbb{NC}}
\def \T {\mathbb T}
\def \CRT {\mathcal T}
\def \U {\mathbb U}

\def \e {\mathrm{e}}
\def \i {\mathrm{i}}
\def \deg {\mathrm{deg}~}
\def \bell {\boldsymbol{\ell}}

\def \R {\mathbb R}
\def \Q {\mathbb Q}
\def \D {\mathbb D}
\def \Db {\overline{\mathbb{D}}}

\def \Z {\mathbb Z}
\def \C {\mathbb C}
\def \S {\mathbb{S}^1}

\def \W {\mathcal{W}}
\def \X {X^{\rm ex}_\alpha}
\def \H {H^{\rm ex}_\alpha}
\def \dim {\mathrm{dim}}					%dim de H
\def \dimMI {\underline{\mathrm{dim}}_{M}}	%dim de M inf
\def \dimMS {\overline{\mathrm{dim}}_{M}}	%dim de M inf

\newcommand{\cv}{\quad\mathop{\longrightarrow}^{}_{n \to \infty}\quad}
\newcommand{\cvloi}{\quad\mathop{\longrightarrow}^{(d)}_{n \to \infty}\quad}

\newcommand{\GW}{Bienaymé--Galton--Watson}

\long\def\symbolfootnote[#1]#2{\begingroup%
\def\thefootnote{\fnsymbol{footnote}}\footnote[#1]{#2}\endgroup}

\title{  \vspace {-2cm}\textbf{Triangulating stable laminations}}
\date{}

\DeclareSymbolFont{extraup}{U}{zavm}{m}{n}
\DeclareMathSymbol{\varheart}{\mathalpha}{extraup}{86}
\DeclareMathSymbol{\vardiamond}{\mathalpha}{extraup}{87}

\makeatletter
\renewcommand*{\@fnsymbol}[1]{\ensuremath{\ifcase#1\or  \spadesuit \or \varheart\or \vardiamond \or \clubsuit \or
   \mathsection\or \mathparagraph\or \|\or **\or \dagger\dagger
   \or \ddagger\ddagger \else\@ctrerr\fi}}
\makeatother

\author{Igor Kortchemski\thanks{CNRS, CMAP, \'Ecole polytechnique, Université Paris-Saclay \hfill  \url{igor.kortchemski@normalesup.org}} 
\qquad \& \qquad Cyril Marzouk\thanks{Institut f\"ur Mathematik, Universit\"at Z\"urich.\hfill  \url{cyril.marzouk@math.uzh.ch}} 
}

\begin{document}

\maketitle

\let\thefootnote\relax\footnotetext{\emph{MSC2010 subject classifications}. Primary 05C80, 60C05; secondary: 05C05, 60J80. }
 
\vspace {-0.8cm}

\begin{abstract} 
We  study the asymptotic behavior of random simply generated noncrossing planar trees in the space of compact subsets of the unit disk, equipped with the Hausdorff distance. Their distributional limits are obtained by triangulating at random the faces of stable laminations, which are random compact subsets of the unit disk made of 
non-intersecting chords coded by stable Lévy processes. We also study other ways  to ``fill-in'' the faces of stable laminations, which leads us to introduce the iteration of  laminations and of  trees.
\end{abstract}

\vspace {-0.3cm}

\section{Introduction}

We are interested in the structure of large random noncrossing trees. By definition, a noncrossing tree with $n$ vertices is a tree drawn in the unit disk of the complex plane having as vertices the $n$-th roots of unity and whose edges are straight line segments which do not cross. The enumeration problem for noncrossing trees was first proposed as Problem E3170 in the American Mathematical Monthly \cite{E3170}. Dulucq \& Penaud \cite{DP93} established a bijection between noncrossing trees with $n$ vertices and ternary trees with $n$ internal vertices, thus showing that there are $ \frac{1}{2n-1} \binom{3n-3}{n-1}$ noncrossing trees with $n$ vertices in another way. Noy \cite{Noy98} pushed forward the enumerative study of noncrossing trees by counting them according to different statistics. Since, various authors have studied  combinatorial and algebraic properties of noncrossing  trees \cite{FN99,DSN02,DN02,PP02,Hou03}. See also \cite{MS09} for motivations from linguistics and proof theory, where noncrossing trees are for instance connected to the number of different readings of an ambiguous sentence. Other families of noncrossing configurations have also attracted some attention \cite{DFHN99,FN99,BPS10,CKdissections}.

 \begin{figure}[!h]
 \begin{center}
    \hfill \includegraphics[width=0.28 \linewidth]{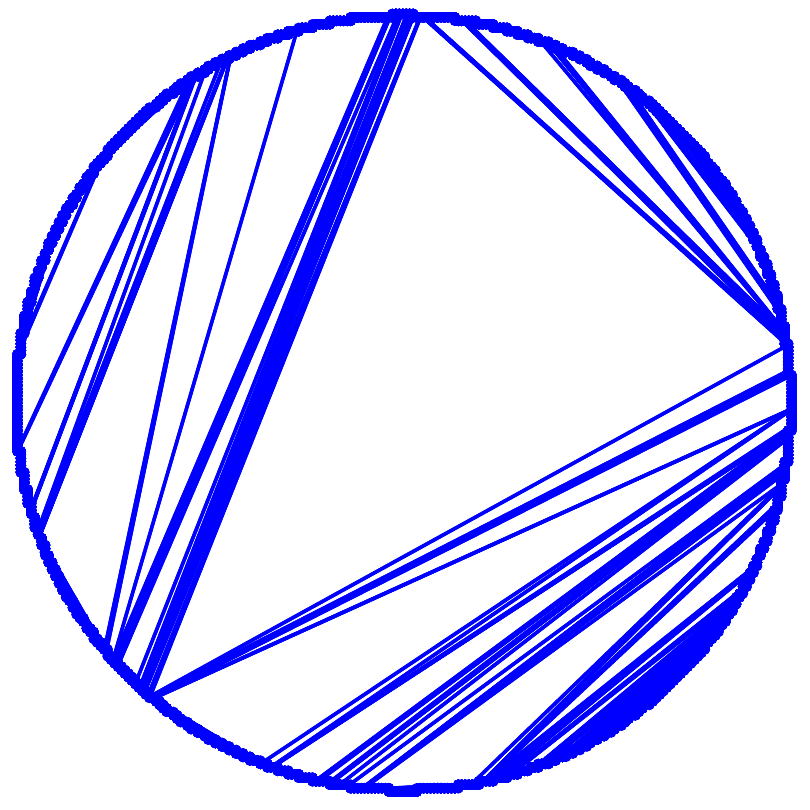}\hfill
    \includegraphics[width=0.28 \linewidth]{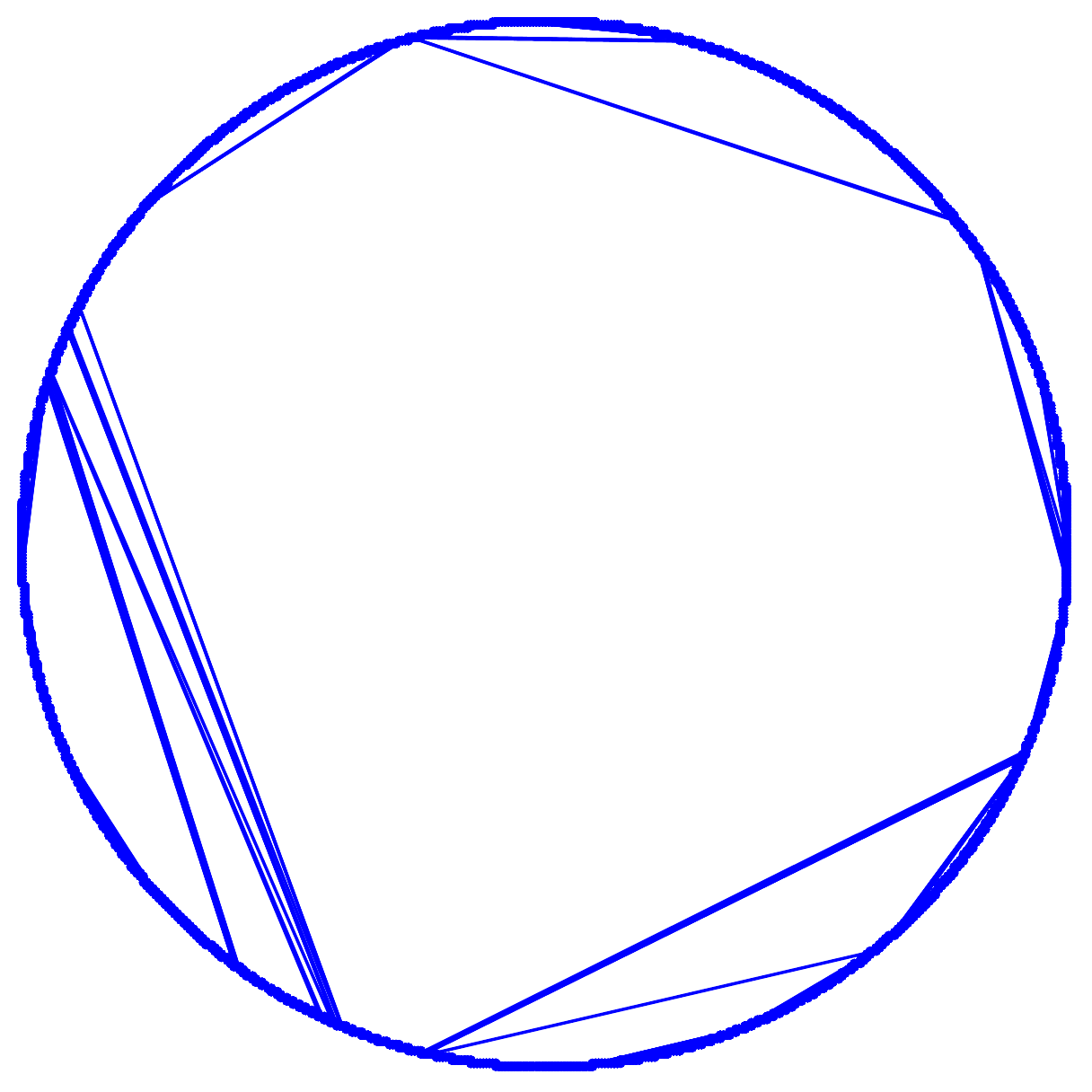}\hfill
  \includegraphics[width=0.28 \linewidth]{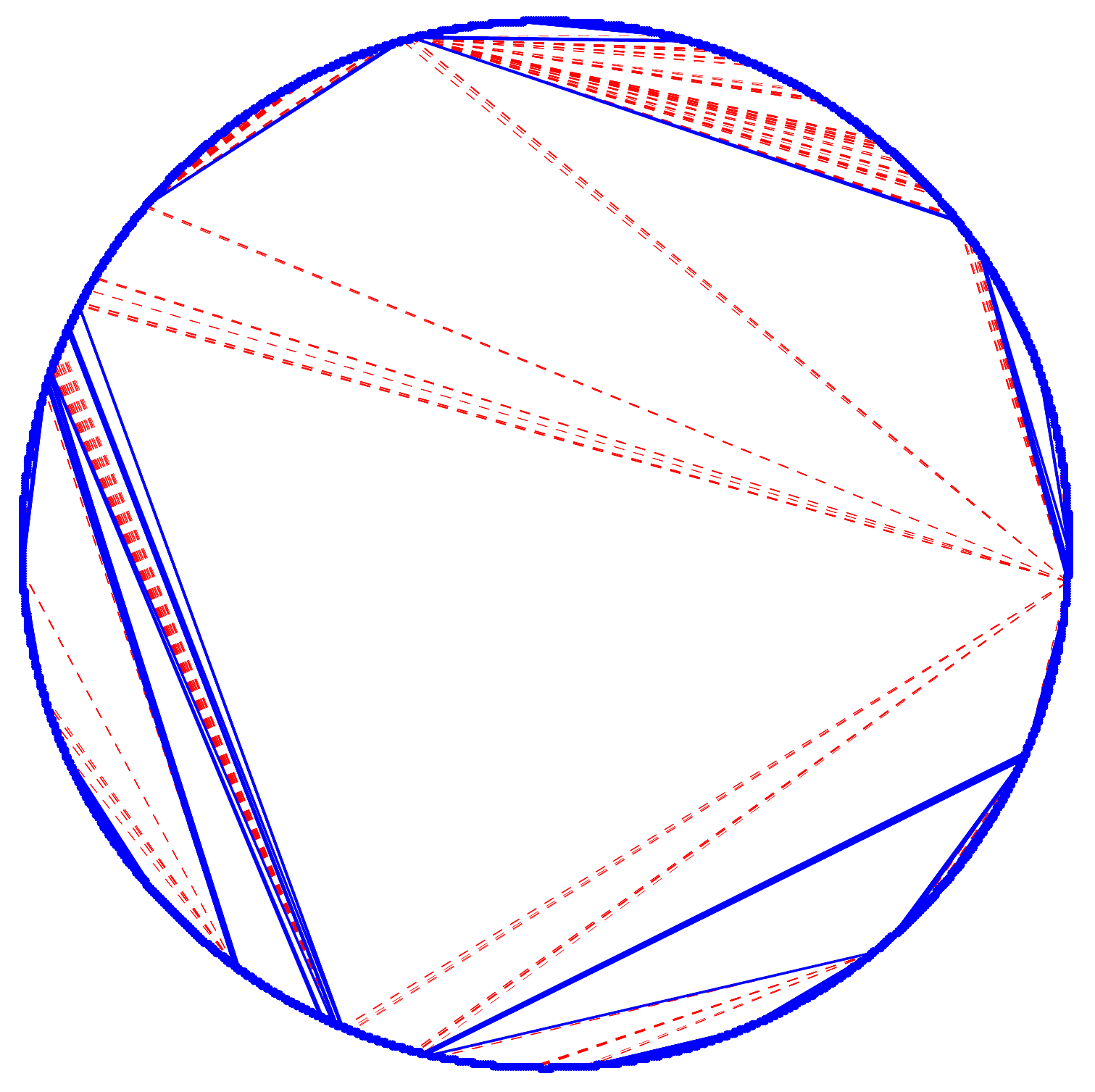}\hfill{}
   \caption{ \label{fig:simu} Simulations from left to right: the Brownian triangulation, an $\alpha=1.1$ stable lamination, and the same lamination with its faces triangulated ``uniformly'' in dashed red.}
 \end{center}
 \end{figure}
 
However, here we study the properties of \emph{random} noncrossing trees. Marckert \& Panholzer \cite{MP02} showed that uniform noncrossing trees on $n$ vertices are almost conditioned Bienaymé--Galton--Watson trees, thus obtaining interesting results concerning the structure of noncrossing trees by using the theory of random plane trees. Later, Curien \& Kortchemski \cite{CKdissections} studied uniform noncrossing trees on $n$ vertices as compact subsets of the unit disk.

In this work, our goal is to consider different ways of choosing noncrossing trees at random, and to study how the geometrical constraint of their planar embeddings influences their structure.

\paragraph{Noncrossing trees seen as subsets of the plane.} Since noncrossing trees are given with a plane embedding, we naturally view them as subsets of the unit disk by considering each edge as a line segment. This idea goes back to Aldous \cite{Ald94b}, who showed that if $P_{n}$ is the regular polygon formed by the $n$-th roots of unit, then, as $n \rightarrow \infty$, a uniform random triangulation of $P_{n}$ converges in distribution in the space of compact subsets of the unit disk equipped with the Hausdorff distance to a random compact subset of the unit disk $\mathbf{L}_2$ called the Brownian triangulation. This set is indeed a triangulation, as its complement in the unit disk is a disjoint union of triangles, and can be built from the Brownian excursion (see Sec.~\ref{sec:trig} below for details). Curien \& Kortchemski \cite{CKdissections} showed that the Brownian triangulation is the universal limit of various classes of uniform random noncrossing graphs built using the vertices of $P_{n}$, such as dissections (which are collections of noncrossing diagonals of $P_{n}$), noncrossing partitions or noncrossing trees. In this spirit, Kortchemski \& Marzouk \cite{KM15} also studied simply generated noncrossing partitions.

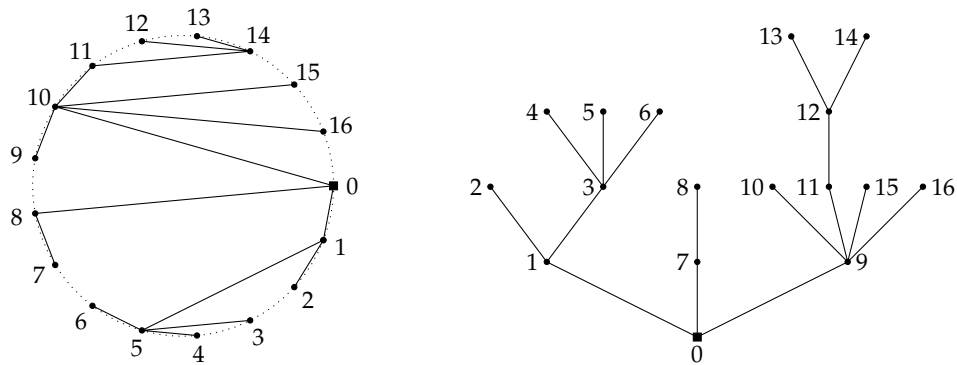
\begin{figure}[ht] \centering
\hfill
\begin{footnotesize}
\hfill
%
%%% LA VERSION NON-CROISÉE
\begin{tikzpicture}%[scale=1.25]
\draw[dotted]	(0,0) circle (2);
\foreach \x in {0, 1, 2, ..., 16}
	\coordinate (\x) at (-\x*360/17 : 2);
\foreach \x in {1, 1, 2, ..., 16}
	\draw[fill=black]	(\x) circle (1pt);
%\foreach \x in {0, 1, 2, ..., 16}
%	\draw[white]	(-\x*360/17 : 2*1.125) node {\x};
\draw
	(-0*360/17 : 2*1.125) node {0}
	(-1*360/17 : 2*1.125) node {1}
	(-2*360/17 : 2*1.125) node {2}
	(-3*360/17 : 2*1.125) node {3}
	(-4*360/17 : 2*1.125) node {4}
	(-5*360/17 : 2*1.125) node {5}
	(-6*360/17 : 2*1.125) node {6}
	(-7*360/17 : 2*1.125) node {7}
	(-8*360/17 : 2*1.125) node {8}
	(-9*360/17 : 2*1.125) node {9}
	(-10*360/17 : 2*1.125) node {10}
	(-11*360/17 : 2*1.125) node {11}
	(-12*360/17 : 2*1.125) node {12}
	(-13*360/17 : 2*1.125) node {13}
	(-14*360/17 : 2*1.125) node {14}
	(-15*360/17 : 2*1.125) node {15}
	(-16*360/17 : 2*1.125) node {16}
;
\draw(0) node [rectangle, scale=.5, fill=black, draw]{};
\draw
%	(0) -- (1)
%	(1) -- (2) (1) -- (3)
%	(3) -- (4)	(3) -- (5)	(3) -- (6)
%	(0) -- (7) -- (8)
%	(0) -- (9)
%	(9) -- (10)	(9) -- (11)	(9) -- (15)	(9) -- (16)
%	(11) -- (12) -- (13)	(12) -- (14)
	(0) -- (1)
	(1) -- (2) (1) -- (5)
	(5) -- (4)	(5) -- (3)	(5) -- (6)
	(0) -- (8) -- (7)
	(0) -- (10)
	(10) -- (9)	(10) -- (11)	(10) -- (15)	(10) -- (16)
	(11) -- (14) -- (13)	(14) -- (12)
;
\end{tikzpicture}
\hfill
%%% L'ARBRE PLANAIRE
\begin{tikzpicture}%[scale=1.25]
\coordinate (0) at (0,0);
	\coordinate (1) at (-2,1);
		\coordinate (11) at (-2.75,2);
		\coordinate (12) at (-1.25,2);
			\coordinate (121) at (-2,3);
			\coordinate (122) at (-1.25,3);
			\coordinate (123) at (-.5,3);
	\coordinate (2) at (0,1);
		\coordinate (21) at (0,2);
	\coordinate (3) at (2,1);
		\coordinate (31) at (1,2);
		\coordinate (32) at (1.75,2);
			\coordinate (321) at (1.75,3);
				\coordinate (3211) at (1.25,4);
				\coordinate (3212) at (2.25,4);
		\coordinate (33) at (2.25,2);
		\coordinate (34) at (3,2);
\draw
	(0) -- (1)
	(1) -- (11) (1) -- (12)
	(12) -- (121)	(12) -- (122)	(12) -- (123)
	(0) -- (2) -- (21)
	(0) -- (3)
	(3) -- (31)	(3) -- (32)	(3) -- (33)	(3) -- (34)
	(32) -- (321) -- (3211)	(321) -- (3212)
;
\draw[fill=black]
%	(0) circle (1pt)
	(1) circle (1pt)
	(11) circle (1pt)
	(12) circle (1pt)
	(121) circle (1pt)
	(122) circle (1pt)
	(123) circle (1pt)
	(2) circle (1pt)
	(3) circle (1pt)
	(32) circle (1pt)
	(321) circle (1pt)
	(21) circle (1pt)
	(31) circle (1pt)
	(33) circle (1pt)
	(34) circle (1pt)
	(3211) circle (1pt)
	(3212) circle (1pt)
;
\draw(0) node [rectangle, scale=.5, fill=black, draw]{};
%
% Labels
\draw
	(0) node[below] {0}
	(1) node[left] {1}
	(11) node[left] {2}
	(12) node[left] {3}
	(121) node[left] {4}
	(122) node[left] {5}
	(123) node[left] {6}
	(2) node[left] {7}
	(21) node[left] {8}
	(3) node[right] {9}
	(31) node[left] {10}
	(32) node[left] {11}
	(321) node[left] {12}
	(3211) node[left] {13}
	(3212) node[left] {14}
	(33) node[right] {15}
	(34) node[right] {16}
;
\end{tikzpicture}
\hfill{}
\end{footnotesize}
\caption{A non-crossing tree with its vertices labelled in clockwise-order and the associated plane tree, called its shape, with its vertices labelled in lexicographical order.}
\label{fig:arbres}
\end{figure}

Kortchemski \cite{Kor11} constructed a one parameter family $\mathbf{L}_\alpha$ of random compact subsets of the unit disk indexed by $\alpha \in (1,2)$  called stable laminations, which are the distributional limits of the more general model of Boltzmann random dissections chosen at random according to certain sequences of weights. Stable laminations are coded by excursions of spectrally positive strictly stable Lévy processes, and unlike the Brownian triangulation, their faces are surrounded by infinitely many chords (see Fig.~\ref{fig:simu} for a simulation and Sec.~\ref{sec:lam} below for details).

\paragraph{Simply generated noncrossing trees.} In this work, we introduce and study the asymptotic behavior of simply generated noncrossing trees in the space of compact subsets of the unit disk equipped with the Hausdorff distance. Given a sequence of non-negative real numbers $(w(k) : k \ge 1)$, we define the weight of a noncrossing tree $\theta$ by
\begin{equation}
\Omega^w(\theta) = \prod_{u \in \theta} w(\deg u).
\end{equation}
Next, for every integer $n \geq 1$, we denote by $\NC_n$ the set of noncrossing trees with $n$ vertices and we set
\begin{equation}
Z_n^w = \sum_{\theta \in \NC_n} \Omega^w(\theta).
\end{equation}
Finally, if $Z_n^w > 0$ (and we will always implicitly restrict our attention to those values of $n$ for which it is the case), we define a probability measure on $\NC_n$ by
\begin{equation}\label{eq:def_arbres_nc_simplement_generes}
\P_n^w(\theta) = \frac{1}{Z_n^w} \Omega^w(\theta) \qquad\text{for all}\quad \theta \in \NC_n.
\end{equation}
A random noncrossing tree sampled according to $\P_n^w$ is called \emph{simply generated}. We choose this terminology because of the similarity with the model of simply generated plane trees, introduced by Meir \& Moon \cite{MM78}. 

For example, if $w \equiv 1$, $\P_n^w$ is the uniform distribution on $\NC_n$. More generally, if $ \mathcal{A}$ is a subset of $\N$ which contains $1$ and if $w(k)= \mathbbm{1}_{ k \in \mathcal{A}}$, then $\P_n^w$ is the uniform distribution on the set of all noncrossing trees with $n$ vertices with all degrees belonging to $ \mathcal{A}$.

\begin{thm}\label{thm:convergence_arbres_SG}
Fix $\alpha \in (1,2]$. There exists a random compact subset of the unit disk, denoted by $\mathbf{L}^{\rm U}_\alpha$, such that the following holds. Let $(w(k) : k \ge 1)$ be a sequence of nonnegative real numbers such that there exists $b > 0$ satisfying
\begin{equation}\label{eq:defb}\sum_{k=0}^\infty (k+1)(k-1) w(k+1) b^k = 0,
\end{equation}
and, moreover, such that the
probability measure
\begin{equation}
\mu(k) = \frac{(k+1) w(k+1) b^k}{\sum_{\ell=0}^\infty (\ell+1) w(\ell+1) b^\ell} \qquad (k \ge 0)
\end{equation}
belongs to the domain of attraction of a stable law of index $\alpha$. If $\mathscr{T}_n$ is a random noncrossing tree sampled according to $\P_n^w$,  then the convergence
\begin{equation}
\mathscr{T}_{n} \cvloi \mathbf{L}^{\rm U}_\alpha
\end{equation}
holds in distribution for the Hausdorff distance on the space of all compact subsets of $\overline{\D}$.
\end{thm}

Recall that a probability distribution $\mu$ belongs to the domain of attraction of a stable law if either it has finite variance (in which case $\alpha=2$), or there exists a slowly varying function $g: \R^{+} \rightarrow  \R^{+}$ such that $\mu([n,\infty))=g(n)n^{-\alpha}$ for $n \geq 1$. See Remark \ref{rem:b} for a probabilistic interpretation of condition \eqref{eq:defb}.

Let us give a rough description of  $\mathbf{L}^{\rm U}_\alpha$. In the case $\alpha=2$, $\mathbf{L}^{\rm U}_2 = \mathbf{L}_2$ is simply Aldous' Brownian triangulation.  However, for $\alpha \in (1,2)$, $\mathbf{L}^{\rm U}_\alpha$ is a triangulation that strictly contains the $\alpha$-stable lamination $\mathbf{L}_\alpha$. Intuitively, $\mathbf{L}^{\rm U}_\alpha$ is constructed  from $\mathbf{L}_\alpha$ by ``triangulating'' each face of $\mathbf{L}_\alpha$ from a uniform random vertex, i.e.~by joining this vertex to each other vertex of the face by a chord. We refer the reader to Fig.~\ref{fig:simu} for a simulation and to Sec.~\ref{sec:trilam} for a precise definition. The random compact set $\mathbf{L}^{\rm U}_\alpha$ is called the \emph{uniform $\alpha$-stable triangulation}. It is interesting to note that unlike the Brownian triangulation or stable laminations, $\mathbf{L}^{\rm U}_\alpha$ is not simply coded by a function as we will see in Remark \ref{rem:codage}.

The main steps to prove Theorem \ref{thm:convergence_arbres_SG} are the following. We first establish deterministic invariance principles in the space of compact subsets of the unit disk (Propositions \ref{prop:invcontinu} and \ref{prop:invcadlag}) for noncrossing trees under conditions involving their \emph{shape}, which is the plane tree structure that they carry (see Fig.~\ref{fig:arbres} for an illustration).  We then establish (Theorem \ref{thm:arbre_SG_presque_GW}) that the shape of  $\mathscr{T}_n$  is a ``modified'' {\GW} tree, where the root has a different offspring distribution, conditioned to have size $n$. This extends a result of Marckert \& Panholzer \cite{MP02} for the uniform distribution. Finally, we show that such trees fulfill the framework of our invariance principles with high probability.

We also compute the Hausdorff dimension of the uniform $\alpha$-stable triangulation.

\begin{thm}\label{thm:dimensions_triangulations}
Fix $\alpha \in (1,2)$ and denote by $A(\mathbf{L}^U_\alpha)$  the set of all end-points of chords in $\mathbf{L}^U_\alpha$. Almost surely,
\begin{equation}
\dim(A(\mathbf{L}^{\rm U}_\alpha)) = \frac{1}{\alpha}
\qquad\text{and}\qquad
 \dim(\mathbf{L}^{\rm U}_\alpha) = 1 + \frac{1}{\alpha}.
\end{equation}
\end{thm}

It is interesting to compare these dimensions with those of stable laminations computed in \cite{Kor11}, which are equal to respectively $1-1/\alpha$ and $2-1/\alpha$. Since $1+1/\alpha>3/2>2-1/\alpha$, the uniform $\alpha$-stable triangulation is ``fatter'' than the Brownian triangulation and any $\beta$-stable lamination.

\paragraph{Applications.} 
An interesting consequence of Theorem \ref{thm:convergence_arbres_SG} is that the geometry of large simply generated noncrossing trees may be very different from that of large simply plane trees generated with the same weights, see Remark \ref{rem:weights}. Theorem \ref{thm:convergence_arbres_SG} also has applications concerning the length of the longest chord of a noncrossing tree. By definition, the (angular) length of a chord  $[\e^{-2i\pi s},\e^{-2i\pi t}]$ with $0 \leq s \leq t \leq 1$ is $\min(t-s,1-t+s)$. Denote by ${\Lambda}(\theta)$ the length of the longest chord of a noncrossing tree $\theta$ and by $ \Lambda(\mathbf{L}^{\rm U}_\alpha)$ the length of the longest chord of $\Lambda(\mathbf{L}^{\rm U}_\alpha)$.

\begin{cor}\label{cor:chorde}Under the assumptions of Theorem \ref{thm:convergence_arbres_SG}, we have
$$\Lambda(\mathscr{T}_{n})  \quad \mathop{\longrightarrow}^{(d)}_{n \rightarrow \infty} \quad \Lambda(\mathbf{L}^{\rm U}_\alpha).$$
\end{cor}
This simply follows from Theorem \ref{thm:convergence_arbres_SG} since the longest chord is a continuous functional for the Hausdorff distance on compact subsets of the unit disk obtained as the union of noncrossing chords. In the case $\alpha=2$, it is known  \cite{Ald94b,DFHN99} that the law of the longest chord of the Brownian triangulation has density
\begin{equation}
\label{eq:cB} \frac{1}{\pi} \frac{3x-1}{ x^{2}(1-x)^{2} \sqrt{1-2x}} \mathbbm{1}_{ \frac{1}{3} \leq x \leq \frac{1}{2}} \ \mathrm{d}x.
\end{equation}
It would be interesting to find an explicit formula for the length of the longest chord of the uniform $\alpha$-stable triangulation for $ \alpha \in (1,2)$. See \cite[Proposition 4.3.]{Shi15} for the expression of the cumulative distribution function of the length of the longest chord in the $\alpha$-stable lamination.

\begin{ex}If $ \mathcal{A}$ is a non-empty subset of $\N$ with $1 \in \mathcal{A}$ and $ \mathcal{A} \neq \{1,2\}$, let $\mathscr{T}_{n}^{ \mathcal{A} }$ be a random noncrossing tree chosen uniformly at random among all those with $n$ vertices and degrees belonging to $ \mathcal{A}$ (provided that they exist). Then $\mathscr{T}_{n}^{ \mathcal{A} }$ converges in distribution to the Brownian triangulation as $n \rightarrow \infty$.
Indeed, this follows from Theorem \ref{thm:convergence_arbres_SG} by taking $w(k)=\mathbbm{1}_{k \in \mathcal{A} }$, as in this case $\mu$ admits finite small exponential moments (since $b<1$, see the beginning of the  proof of Theorem \ref{thm:enumA} below).  Theorem \ref{thm:convergence_arbres_SG} thus extends Theorem 3.1 in \cite{CKdissections}, which shows the convergence to the Brownian triangulation of large uniform noncrossing trees. Also, by Corollary \ref{cor:chorde}, the length of the longest chord of $\mathscr{T}_{n}^{ \mathcal{A} }$ converges in distribution to the random variable whose law is given by \eqref{eq:cB}.  It is remarkable that this limiting distribution does not depend on $ \mathcal{A}$.
\end{ex}

\paragraph{Degree-constrained noncrossing trees.} Let $\mathcal{A} \subset \N$ be a non-empty subset with $1 \in \mathcal{A}$. We let $\NC_{n}^{\mathcal{A}}$ be the set of all noncrossing trees having $n$ vertices and with degrees only belonging to $ \mathcal{A}$. As an application of our techniques, we establish the following enumerative result.

\begin{thm}\label{thm:enumA}
Assume that $ \mathcal{A} \neq \{1,2\}$. Let $b>0$ be such that $\sum_{k+1 \in \mathcal{A} }(k+1)(k-1) b^k = 0$ and define
$$K_{\mathcal{A}} \quad  \coloneqq \quad \gcd( \mathcal{A}-1) \cdot  \sqrt{ \frac{ \sum_{k+1 \in \mathcal{A} } (k+1)b^{k}}{ 2 \pi \sum_{k+1 \in \mathcal{A}} (k+1)(k^{2}-1)b^{k}}} \cdot \left( \sum_{k \in \mathcal{A}} k b^{k} \right).$$
We have
$$ \# \NC_{n}^{\mathcal{A}}   \quad\mathop{\sim}_{n \rightarrow \infty}\quad  K_{\mathcal{A}} \cdot  \left( \sum_{k+1 \in \mathcal{A}}(k+1)b^{k-1} \right)^{n-1} \cdot  n^{-3/2},$$
where the limit is taken along the subsequence of those values of $n$ for which $\NC_{n}^{\mathcal{A}}  \neq \varnothing$.
\end{thm}

 We  give a simple proof of this  by using the probabilistic structure of simply generated non-crossing trees. For example, if $ \mathcal{A}=\mathbb{N}$, one finds that  $ \# \NC_{n} \sim  ({9 \sqrt{3\pi}})^{-1} \cdot (27/4)^{n} \cdot n^{-3/2}$ as $n \rightarrow \infty$, which is consistent with the fact that $ \# \NC_{n} =\frac{1}{2n-1} \binom{3n-3}{n-1}$.

 \begin{figure}[!h]
 \begin{center}
    \hfill \includegraphics[width=0.4 \linewidth]{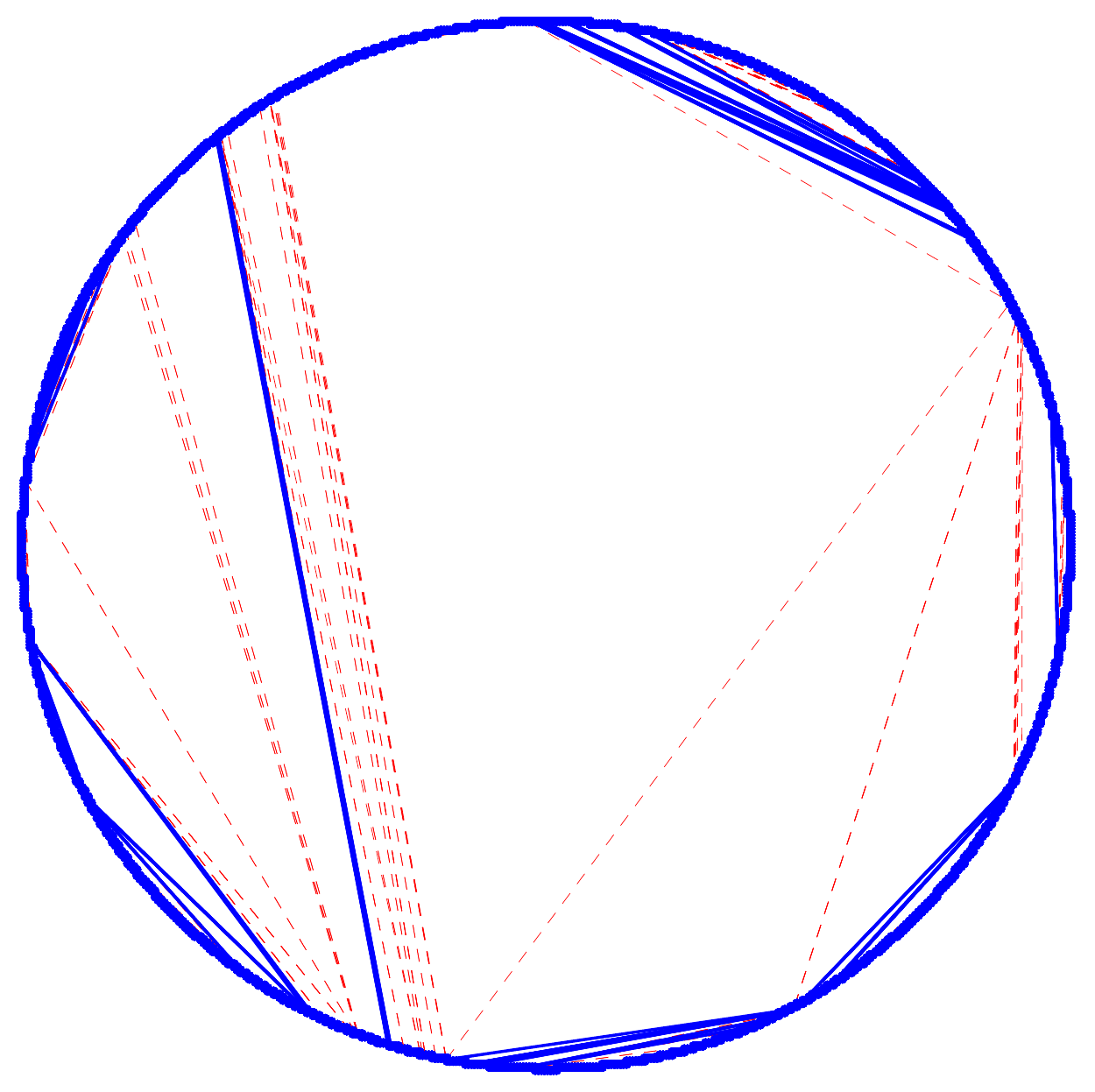}\hfill
    \includegraphics[width=0.4 \linewidth]{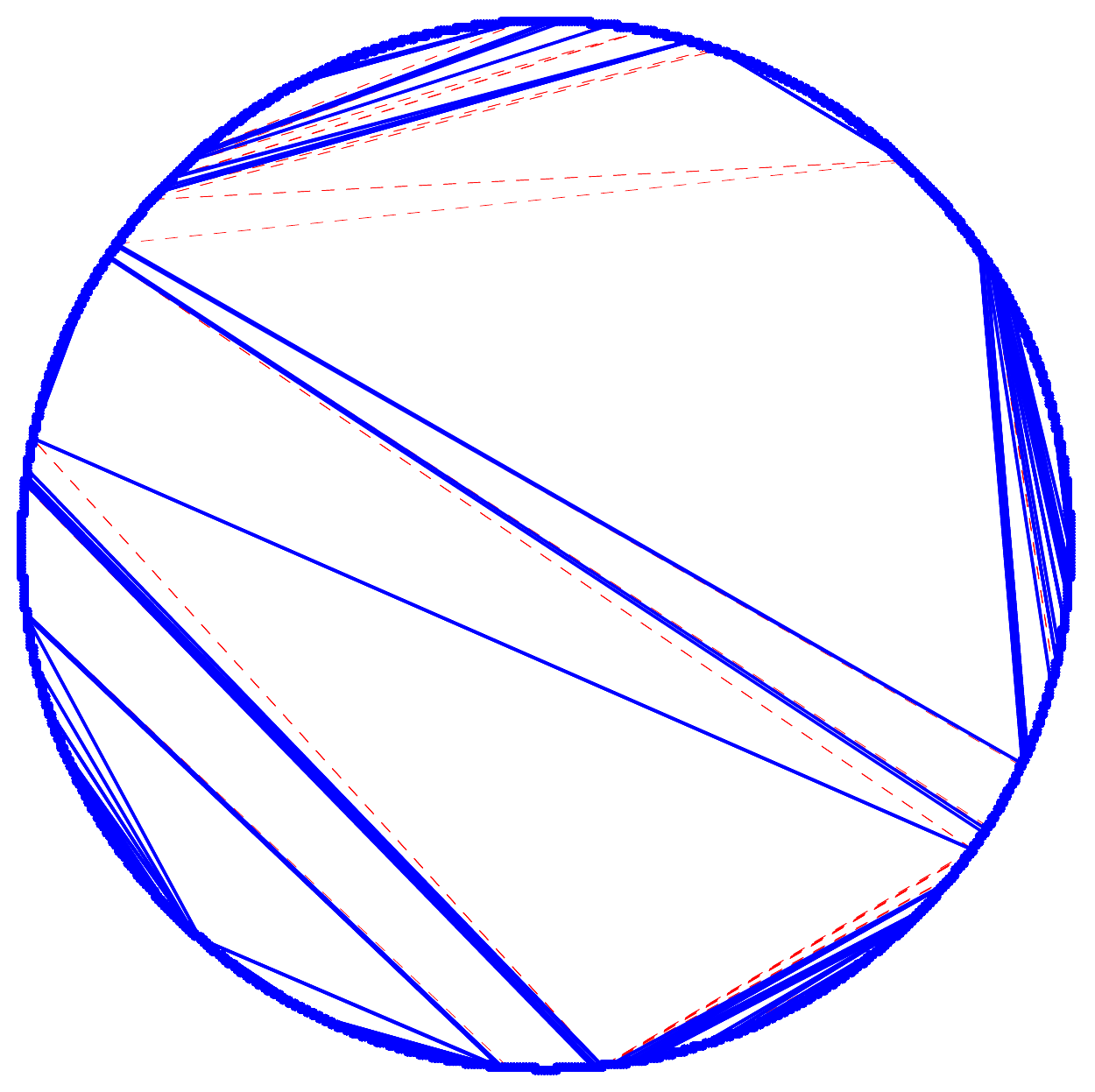}\hfill{}
   \caption{ \label{fig:iterate} Simulations from left to right: $\beta=1.4$ laminations iterated inside an $\alpha=1.1$, and  $\beta=1.1$ laminations iterated inside an $\alpha=1.4$. The chords of the $\beta$-stable laminations are in dashed red.}
 \end{center}
 \end{figure}

\paragraph{Iterating laminations.}
The random set  $\mathbf{L}^{\rm U}_\alpha$ is constructed from an $\alpha$-stable lamination $\mathbf{L}_\alpha$ by triangulating independently each face of $\mathbf{L}_\alpha$. More generally, one can consider independent random $\beta$-laminations in each face of $\mathbf{L}_\alpha$ (see Fig.~\ref{fig:iterate} for an illustration). We can also iterate this procedure: fix a sequence $(\alpha_k : k \ge 1)$ with values in $(1,2)$, let $\mathbf{L}^{(0)}$ be the unit circle and define next recursively for $n \ge 1$ random sets $\mathbf{L}^{(n)}$ by sampling independently an $\alpha_n$-stable lamination in each face of $\mathbf{L}^{(n-1)}$. We give a formal definition of this procedure in Sec.~\ref{sec:extensions}, with several possible further directions of research concerning the study of $\mathbf{L}^{(n)}$.

\paragraph{Acknowledgments.} I. K.~acknowledges partial support from Agence Nationale de la Recherche, grant number ANR-14-CE25-0014 (ANR GRAAL), and from the ``City of Paris, grant Emergences Paris 2013, Combinatoire à Paris''. C. M.~acknowledges support from the Swiss National Science Foundation 200021\_144325/1.

\tableofcontents

\section{Coding plane trees and noncrossing trees}
\label{sec:codings}

We start by explaining how we code plane trees and noncrossing trees. These codings are also useful to understand the intuition hiding behind the definitions of their continuous analogs.

\subsection{Plane trees}
\label{sec:planetrees}

%%%
\paragraph{Definitions.} We use Neveu's formalism \cite{Nev86} to define plane trees: let $\N = \{1, 2, \dots\}$ be the set of all positive integers, set $\N^0 = \{\varnothing\}$ and consider the set of labels $\U = \bigcup_{n \ge 0} \N^n$. For $u = (u_1, \dots, u_n) \in \U$, we denote by $|u| = n$ the length of $u$; if $n \ge 1$, we define $pr(u) = (u_1, \dots, u_{n-1})$ and for $i \ge 1$, we let $ui = (u_1, \dots, u_n, i)$; more generally, for $v = (v_1, \dots, v_m) \in \U$, we let $uv = (u_1, \dots, u_n, v_1, \dots, v_m) \in \U$ be the concatenation of $u$ and $v$. We endow $\U$ with the lexicographical order: given $v,w \in \U$, let $z \in \U$ be their longest common prefix, that is $v = z(v_1, \dots, v_n)$, $w = z(w_1, \dots, w_m)$ and $v_1 \ne w_1$, then $v \prec w$ if $v_1 < w_1$.

A plane tree is a nonempty finite subset $\tau \subset \U$ such that (i) $\varnothing \in \tau$; (ii) if $u \in \tau$ with $|u| \ge 1$, then $pr(u) \in \tau$; (iii)  if $u \in \tau$, then there exists an integer $k_u(\tau) \ge 0$ such that $ui \in \tau$ if and only if $1 \le i \le k_u(\tau)$.

We will view each vertex $u$ of a tree $\tau$ as an individual of a population for which $\tau$ is the genealogical tree. For $u,v \in \tau$, we let  $\llbracket u, v \rrbracket$ be the vertices belonging to the shortest path from $u$ to $v$. The vertex $\varnothing$ is called the root of the tree and for every $u \in \tau$, $k_u(\tau)$ is the number of children of $u$ (if $k_u(\tau) = 0$, then $u$ is called a leaf, otherwise, $u$ is called an internal vertex), $|u|$ is its generation, $pr(u)$ is its parent and more generally, the vertices $u, pr(u), pr \circ pr (u), \dots, pr^{|u|}(u) = \varnothing$ belonging to $ \llbracket \varnothing, u \rrbracket$ are its ancestors. To simplify, we will sometimes write $k_{u}$ instead of $k_{u}(\tau)$. We denote by $\T$ the set of all plane trees and for each integer $n \geq 1$, by $\T_n$ the set of plane trees with $n$ vertices.

\paragraph{Bienaymé--Galton--Watson trees.} Let $\mu$ be a critical probability measure on $\Z_+$, by which we mean that $\mu(0) > 0$, $\mu(0)+\mu(1)<1$ (to avoid trivial cases) and with expectation $\sum_{k = 0}^\infty k \mu(k) = 1$. The law of a {\GW} tree with offspring distribution $\mu$ is the unique probability measure $\mathrm{BGW}^\mu$ on $\T$ such that for every $\tau \in \mathbb{T}$,
\begin{equation}\label{eq:def_GW}
\mathrm{BGW}^\mu(\tau) = \prod_{u \in \tau} \mu(k_u).
\end{equation}
For each integer $n \ge 1$, we denote by $\mathrm{BGW}^\mu_n$ the law of a {\GW} tree with offspring distribution $\mu$ conditioned to have $n$ vertices; we shall always implicitly restrict ourselves to the values of $n$ such that the conditioning makes sense.
%%%

 \paragraph{Coding by the {\L}ukasiewicz path.} Fix a tree $\tau \in \T_n$ and let $\varnothing = u(0) \prec u(1) \prec \dots \prec u(n-1)$ be its vertices, listed in lexicographical order.  The {\L}ukasiewicz path $\W(\tau)= (\W_j (\tau): 0 \le j \le n)$ of $\tau$  is defined by $\W_0(\tau) = 0$ and for every $0 \le j \le n-1$,
\begin{equation}
\W_{j+1}(\tau) = \W_j(\tau) + k_{u(j)}(\tau)-1.
\end{equation}
One easily checks (see e.g.~\cite{LG05}) that $\W_j(\tau) \ge 0$ for every $0 \le j \le n-1$ but $\W_n(\tau)=-1$. Observe that $\W_{j+1}(\tau) - \W_j(\tau) \ge -1$ for every $0 \le j \le n-1$, with equality if and only if $u(j)$ is a leaf of $\tau$. We shall think of such a path as the step function on $[0,n]$ given by $s \mapsto \W_{\lfloor s \rfloor}(\tau)$.

\paragraph{Scaling limits.} Fix $\alpha \in (1, 2]$ and consider a strictly stable spectrally positive L{\'e}vy process of index $\alpha$: $X_\alpha$ is a random process with paths in the set $\D([0, \infty), \R)$ of c{\`a}dl{\`a}g functions endowed with the Skorokhod $J_{1}$ topology (see e.g. Billingsley \cite{Bil99} for details) which has independent and stationary increments, no negative jump and such that
$\Es{\exp(-\lambda X_\alpha(t))} = \exp(t \lambda^\alpha) $ for every $t, \lambda > 0$. Using excursion theory, it is then possible to define $\X$, the normalized excursion of $X_\alpha$, which is a random variable with values in $\D([0, 1], \R)$, such that $\X(0) = \X(1) = 0$ and, almost surely, $\X(t) > 0$ for every $t \in (0,1)$. We do not enter into details and refer to Bertoin \cite{Ber96} for background.

An important point is that $\X$ is continuous for $\alpha=2$, and indeed $X^{\rm ex}_2 / \sqrt{2}$ is the standard Brownian excursion, whereas the set of discontinuities of $\X$ is dense in $[0,1]$ for every $\alpha \in (1,2)$.

Duquesne \cite{Du03}  (see also \cite{K11}) provides the following limit theorem which is the steppingstone of our convergence results. Let $\alpha \in (1,2]$ and $\mu$ a critical probability measure on $\Z_+$ in the domain of attraction of a stable law of index $\alpha$. For every $n \ge 1$ for which $\mathrm{BGW}^\mu_n$ is well defined, sample $\CRT_n$ according to $\mathrm{BGW}_n^\mu$. Then there exists a sequence $(B_n)_{n \ge 1}$ of positive constants satisfying $\lim_{n \to \infty} B_n = \infty$, such that the convergence \begin{equation}\label{eq:Duq}
\left( \frac{1}{B_n} \W_{\lfloor ns \rfloor}(\CRT_n) : s \in [0,1]\right)  \quad \mathop{\longrightarrow}^{(d)}_{n \rightarrow \infty} \quad 
(\X(s) : s \in [0,1])
\end{equation}
holds in distribution in the space $\D([0, 1], \R)$.

The sequence $(B_{n})$ is regularly varying with index $1/\alpha$, meaning that  if $(u_{n})_{n \geq 1}$ and $(v_{n})_{n \geq 1}$ are two sequences of integers tending to $\infty$ and such that $u_{n}/v_{n} \rightarrow s >0$, then $B_{u_{n}}/B_{v_{n}} \rightarrow s^{1/\alpha}$ as $n \rightarrow \infty$, and may be chosen to be increasing (see e.g.~\cite[Theorem 1.10]{Kor12}, which also gives the dependence of $B_{n}$ in terms of $\mu$). When $\mu$ has finite positive variance $\sigma^{2}$, one can take $B_{n}=\sigma \sqrt{n/2}$.

\subsection{Noncrossing trees}

Let $\tau \in \T_{n}$ be a plane tree with $n$ vertices with its vertices $\varnothing = u(0) \prec u(1) \prec \dots \prec u(n-1)$ listed in lexicographical order. We set 
$$\mathbb{C}(\tau)= \{(l_{1},l_{2}, \ldots, l_{n-1}): 0 \leq l_{j} \leq k_{u(j)}(\tau) \textrm{ for every } 1 \leq j \leq n-1 \}$$
and
$$\T_{n}^{\mathsf{dec}}= \{ (\tau, \mathbf{c}) : \tau \in \T_{n} \textrm{ and } \mathbf{c} \in \mathbb{C}(\tau)\}.$$
Elements of $\T_{n}^{\mathsf{dec}}$ are called decorated trees, and we can view $l_{j}$ as the label carried by the vertex $u(j)$. Note that $\#\mathbb{C}(\tau) = \prod_{u \in \tau \setminus \{\varnothing\} }(k_{u}(\tau)+1)$ for every $\tau \in \mathbb{T}$.

If $\theta$ is a noncrossing tree, we let $S(\theta)$ be its shape, which is the plane tree associated with $\theta$ and rooted at the vertex corresponding to the complex number $1$ (see Fig.~\ref{fig:arbres} for an example). If $\theta$ is a noncrossing tree with $n$ vertices and $\varnothing = u(0) \prec u(1) \prec \dots \prec u(n-1)$ are the vertices of its shape listed in lexicographical order, for every $1 \leq i \leq n-1$, we let ${L}_{i}(\theta)$ be the number of children of $u(i)$ lying to the ``left'' of $u(i)$ (that is lying in the left half-plane formed by the line joining $u(i)$ with  the complex number $1$), and set
$${C}(\theta)=({L}_{1}(\theta),{L}_{2}(\theta), \ldots,{L}_{n-1}(\theta))  \quad  \in  \quad \mathbb{C}(\theta).$$

The following result is a reformulation of the ``left-right'' coding of noncrossing trees in \cite{PP02}.

\begin{prop}\label{prop:bij}For every $n \geq 1$, the mapping
\begin{align*}
{\Phi}_{n}: \quad  \NC_{n}& \quad \longrightarrow \quad  \T_{n}^{\mathsf{dec}}\\
\theta & \quad  \longmapsto  \quad ( S(\theta),{C}(\theta))
\end{align*}
is a bijection.
\end{prop}

\begin{proof}We describe the reverse map $\Phi_{n}^{-1}$; this will also be useful later. Fix $(\tau,(l_{1},l_{2}, \ldots, l_{n-1})) \in \T_{n}^{\mathsf{dec}}$.  Let $\varnothing = u(0) \prec u(1) \prec \dots \prec u(n-1)$ be the vertices of $\tau$ labelled in lexicographical order. To simplify notation, for every $u \in \tau$ with $u \neq \varnothing$, we set $n(u)=k$ if $u$ is the $k$-th child of its parent and we let $l(u)$ be the label carried by $u$, that is $l(u)=l_{j}$ if $u=u(j)$. Then, for every $u \in \tau$,  set
$$L(u)= \# \left\{v \in \rrbracket \varnothing, u \rrbracket : |v| \ge 2 \textrm{ and } n(v) \leq l(pr(v))   \right\}, \qquad R(u)=|u|-L(u)-1,$$
where we recall that $pr(v)$ is the parent of $v$. 
Intuitively speaking, $L(u)$ and $R(u)$ represent  the number of vertices of $\rrbracket \varnothing, u \rrbracket$ that will be respectively folded to the left and to the right of $u$ in the associated noncrossing tree which is defined as follows.

First map $\varnothing$ to the complex number $1$. Then, for every $ 1 \leq p \leq n-1$, let $k_{p}$ be the number of children of $u(p)$. If $k_{p}=0$,  map $u(p)$ to $\e^{-2 \i \pi \cdot (p-R(u(p)))/n}$. Otherwise, for $1 \leq i \leq k_{p}$, let
$T_{i}$ be the size of the subtree grafted on the $i$-th child of $u(p)$ (so that $T_{i}$ is the number of its non strict descendants) with the convention $T_{0}=0$. Then map $u(p)$ to $\e^{-2 \i \pi \cdot (p-R(p)+ T_{1}+T_{2}+\cdots+T_{l_{p}})/n}$. It is then a simple matter to check that $\Phi_{n} \circ \Phi_{n}^{-1}$ and $\Phi_{n}^{-1} \circ \Phi_{n}$ are the identity, which completes the proof. \end{proof}

In Section \ref{sec:inv}, we give sufficient conditions on a sequence $(\tau_{n}^{\mathrm{dec}})_{n \geq 1}$ of decorated trees which ensure that the associated noncrossing trees $\Phi_{n}^{-1}(\tau_{n}^{\mathrm{dec}})$ converge to triangulated laminations, which form a family of compact subsets of the unit disk which we now define.

\section{Triangulations, laminations and triangulated laminations}

We denote by $\Db= \{z \in \C : |z| \leq 1\} $ the closed unit disk. A geodesic lamination of $\Db$ is a closed subset of $\Db$ which can be written as the union of a collection of noncrossing chords. In the sequel, by lamination we will always mean geodesic lamination of $\Db$. A lamination is said to be maximal when it is maximal for the inclusion relation among laminations. We call faces of a lamination the connected components of its complement in $\overline{\D}$; note that the faces of a maximal lamination are open triangles whose vertices belong to $\S$, a maximal lamination is also called a triangulation.

\subsection{Triangulations coded by continuous functions}
\label{sec:trig}

Let $f:[0,1] \rightarrow \R_{+}$ be a continuous function with $f(0)=f(1)=0$ and such that the following assumption \eqref{H_f} holds:

\begin{equation}\tag{$H_{f}$}\label{H_f}
\text{The local minima of } f \text{ are distinct}.
\end{equation}
This means that if $0 \leq a<b<c<d \leq 1$ are such that the infimum of $f$ over $]a,b[$ is attained at a point of $]a,b[$, and that over $]c,d[$ is attained at a point of $]c,d[$ as well, then $\min_{]a,b[}f \neq \min_{]c,d[} f$.

We define an equivalence relation on $[0,1]$ by $s \overset{f}{\thicksim} t$ whenever $f(s) = f(t) = \min_{[s \wedge t, s \vee t]} f$. We then define a subset of $\overline{\D}$ by
\begin{equation}\label{eq:intro_def_triangulation_brownienne}
\mathbf{L}(f) \quad \coloneqq \quad \bigcup_{s \overset{f}{\thicksim} t} \left[\e^{-2\i\pi s}, \e^{-2\i\pi t}\right].
\end{equation}
Using the fact that $f$ is continuous and its local minima are distinct, one can prove (see e.g. \cite[Prop. 2.1]{LGP08}) that $\mathbf{L}(f)$ is a geodesic lamination of $\overline{\D}$. Furthermore, it is maximal for the inclusion relation among geodesic laminations of $\overline{\D}$.  For this reason, we say that $\mathbf{L}(f)$ is the \emph{triangulation coded by $f$}. 

Now let $\mathbbm{e} =X^{\rm ex}_2$ be $\sqrt{2}$ times the standard Brownian excursion. Since $\mathbbm{e}$ has almost surely distinct local minima, the lamination $\mathbf{L}(\mathbbm{e})$ is maximal, it is called the Brownian triangulation and is also denoted by $\mathbf{L}_{2}$. This set has been introduced by Aldous \cite{Ald94b}.

\subsection{Laminations coded by càdlàg functions}
\label{sec:lam}

Recall that  ${\D}([0, 1], \R)$  is the space of real-valued càdlàg functions on $[0,1]$ equipped with the Skorokhod $J_{1}$ topology. If $X \in {\D}([0, 1], \R)$, we set $\Delta X(t)=X(t)-X(t-)$ for $t \geq 0$, with the convention $X(0-)=X(0)$. We fix a function $Z \in {\D}([0, 1], \R)$ such that $Z(0)=Z(1) = 0$, $Z(t) > 0$ and $\Delta Z(t) \ge 0$ for every $t \in (0,1)$, and satisfying the following four properties:
\begin{enumerate}[label=\color{blue}(H\arabic*)]
\item\label{H1}
For every $0 \le s < t \le 1$, there exists at most one value $r \in (s, t)$ such that $Z(r) = \inf_{[s, t]} Z$.
\item\label{H2}
For every $t \in (0,1)$ such that $\Delta Z(t) > 0$, we have $\inf_{[t, t+\varepsilon]} Z < Z(t)$ for every $0 < \varepsilon \le 1-t$;
\item\label{H3}
For every $t \in (0,1)$ such that $\Delta Z(t) > 0$, we have $\inf_{[t-\varepsilon, t]} Z < Z(t-)$ for every $0 < \varepsilon \le t$;
\item\label{H4}
For every $t \in (0,1)$ such that $Z$ attains a local minimum at $t$ (which implies $\Delta Z(t) = 0$), if $s = \sup\{ u \in [0,t] : Z(u) < Z(t)\}$, then $\Delta Z(s) > 0$ and $Z(s-) < Z(t) < Z(s)$.
\end{enumerate}
We recall the construction in \cite{Kor11} of a lamination $L(Z)$ from $Z$. To this end, we define a relation (not equivalence relation in general) on $[0, 1]$ as follows: for every $0 \le s < t \le 1$, we set
\begin{equation}
s \simeq^Z t \qquad\text{if}\qquad t = \inf \left\{u > s : Z(u) \le Z(s-) \right\},
\end{equation}
then for $0 \le t < s \le 1$, we set $s \simeq^Z t$ if $t \simeq^Z s$, and we agree that $s \simeq^Z s$ for every $s \in [0,1]$. We finally define a subset of $\overline{\D}$ by
\begin{equation}\label{eq:lamination_stable}
L(Z) \quad \coloneqq \quad \bigcup_{s \simeq^Z t} \left[\e^{-2\i\pi s}, \e^{-2\i\pi t}\right].
\end{equation}
Using the four above properties, it is proved in \cite[Prop. 2.9]{Kor11} that $L(Z)$ is a geodesic lamination of $\overline{\D}$, called the lamination coded by $Z$.

Recall that $\X$ denotes the normalized excursion of a spectrally positive strictly stable Lévy process for $ \alpha \in (1,2]$. For every $\alpha \in (1, 2)$, $\X$ fulfills the above properties with probability one (\cite[Proposition~2.10]{Kor11}), we can therefore set
\begin{equation}
\mathbf{L}_\alpha \quad  \coloneqq \quad  L(\X),
\end{equation}
which is called the stable lamination of index $\alpha$.

We recall from \cite[Proposition 3.10]{Kor11} the description of the faces of $L(Z)$ (this reference actually only covers the case where $Z=\X$, but the arguments carry out in this setting as well), which are the connected components of the complement of $L(Z)$ in $\D$.
The faces of $L(Z)$ are in one-to-one correspondence with the jump times of $Z$ (observe that the latter set is countable since $Z$ is c{\`a}dl{\`a}g). For every $s,t \in (0,1)$, let $\mathbb{H}(s,t)$ be the open half-plane bounded by the line containing $\e^{-2\i\pi s}$ and $\e^{-2\i\pi t}$, which does not contain the complex number $1$.
Then for every jump time $s$ of $Z$, letting $t = \inf\{u > s : Z(u) = Z(s-)\}$, the face $V_{s}$ of $L(Z)$ associated with $s$ is the unique one contained in $\mathbb{H}(s,t)$ whose boundary contains the chord $[\e^{-2\i\pi s}, \e^{-2\i\pi t}]$.
Moreover, the ``boundary'' of the face $V_{s}$ which belongs to $\S$ is given by
\begin{equation}\label{eq:sommets_face}
B_{s} \quad \coloneqq \quad \overline{V} \cap \S =  \left\{r \in [s,t] : Z(r) = \inf_{[s,r]}Z \right\},
\end{equation}
where we identify the interval $[0, 1)$ with the circle $\S$ via the mapping $t \mapsto \e^{-2\i\pi t}$ to ease notation.

\subsection{Triangulated laminations}
\label{sec:trilam}

We next define triangulations which are, informally, obtained from $L(Z)$ by ``triangulating'' all its faces, i.e.~for each face of $L(Z)$ we choose a special vertex on its boundary on $\S$ and join it to all the other vertices of this face by chords.

Fix $Z \in \D([0,1],\R)$ satisfying  \ref{H1}, \ref{H2}, \ref{H3}, \ref{H4}. Let $J(Z)= \{u \in [0,1] : \Delta Z(u)>0\}$  be the set of all jump times of $Z$, and let $\boldsymbol{\ell}=( \ell_{u}; u \in J(Z))$ be a sequence of nonnegative real numbers indexed by these jump times such that $ 0 \leq \ell_{u} \leq 1$ for every $u \in J(Z)$. By convention, we shall always assume that $\ell_{u}=0$ if $u \not \in J(Z)$. The sequence $\bell$ will be called a \emph{jumps labelling}.

For every $u \in J(Z)$, set \begin{equation}
\label{eq:pu}p_{u}(\boldsymbol{\ell})= \inf  \left\{r \geq u : Z_{r}=Z_{u} - \Delta Z (u) \cdot \ell_{u} \right\} \qquad \textrm{and} \qquad C_{u}(\boldsymbol{\ell})= \bigcup_{r \in B_{u}} \left[\e^{-2\i\pi p_{u}(\boldsymbol{\ell})}, \e^{-2\i\pi r}\right],
\end{equation}
where we recall that $B_{u}$ is defined by \eqref{eq:sommets_face}. Note that $p_{u}(\bell) \in B_{u}$ for every $u \in J(Z)$.
Finally define
\begin{equation}
\label{eq:triangulation_stable_canonique}L(Z,\boldsymbol{\ell}) \quad \coloneqq \quad   L(Z) \cup \bigcup_{s \in J(Z)}  C_{s}(\boldsymbol{\ell}).
\end{equation}
Intuitively speaking, $L(Z,\boldsymbol{\ell})$ is obtained from $L(Z)$ by triangulating each face as follows: inside every face  $V_{s}$ of  $L(Z)$ indexed by a jump time $s$, choose a special vertex on its boundary $B_{s}$  indexed by $p_{s}(\bell)$, and draw chords from this special vertex to  all the other points of $B_{s}$. The point is that the latter set is uncountable, so some care is needed to define the special vertex, hence the purpose of the jumps labelling $\bell$. Roughly speaking, $ x \in [0,1] \mapsto  \inf  \left\{u \geq s : Z_{u}=Z_{s} - \Delta Z (s) \cdot x \right\}$ plays the role of the inverse of the local time of vertices of $B_{s}$ (that is a measurement of the evolution of ``number'' of vertices of $B_{s}$ as one goes around $\S$) and allows to identify $[0,1]$ with $B_{s}$.

\begin{prop}\label{prop:triangulation_stable_canonique}
Under the assumptions \ref{H1}, \ref{H2}, \ref{H3}, \ref{H4}, for every jumps labelling $\bell$, the set $L(Z, \bell)$ is a triangulation of $\overline{\D}$.
\end{prop}

\begin{proof}
First note that the chords defining $L(Z,\boldsymbol{\ell})$ in \eqref{eq:triangulation_stable_canonique} are noncrossing: there exists no $4$-tuple $0 \le s < s' < t < t' \le 1$ such that both chords $[\e^{-2\i\pi s}, \e^{-2\i\pi t}]$ and $[\e^{-2\i\pi s'}, \e^{-2\i\pi t'}]$ belong to $L(Z,\boldsymbol{\ell})$. Indeed, suppose there exists such a $4$-tuple. Clearly, we cannot have $[\e^{-2\i\pi s}, \e^{-2\i\pi t}] \subset C_{u}(\boldsymbol{\ell})$ and $[\e^{-2\i\pi s'}, \e^{-2\i\pi t'}] \subset C_{u}(\boldsymbol{\ell})$ for any $u \in J(Z)$ and neither do we have $[\e^{-2\i\pi s}, \e^{-2\i\pi t}] \subset L(Z)$ and $[\e^{-2\i\pi s'}, \e^{-2\i\pi t'}] \subset L(Z)$ since $L(Z)$ is a lamination.

Assume next that $[\e^{-2\i\pi s}, \e^{-2\i\pi t}] \subset C_{u}(\boldsymbol{\ell})$ for a certain $u \in J(Z)$ and $[\e^{-2\i\pi s'}, \e^{-2\i\pi t'}] \subset L(Z)$; then $s,t \in B_u$ so $u \le s < s' < t < t'$ and $Z(t) = \inf_{[u,t]} Z$. It follows that $Z(t) \le Z(s'-)$ which contradicts $t' = \inf\{r > s' : Z(r) \le Z(s'-)\}$. The case $[\e^{-2\i\pi s'}, \e^{-2\i\pi t'}] \subset C_{u}(\boldsymbol{\ell})$ for a certain $u \in J(Z)$ and $[\e^{-2\i\pi s}, \e^{-2\i\pi t}] \subset L(Z)$ yields a similar contradiction.

The last case to consider is $[\e^{-2\i\pi s}, \e^{-2\i\pi t}] \subset C_{u}(\boldsymbol{\ell})$ for a certain $u \in J(Z)$ and $[\e^{-2\i\pi s'}, \e^{-2\i\pi t'}] \subset C_{u'}(\boldsymbol{\ell})$ for a certain $u' \in J(Z)$ with $u' \ne u$. Let $v = \inf\{r > u : Z(r) = Z(u-)\}$ and $v' = \inf\{r > u' : Z(r) = Z(u'-)\}$; then $u \le s < t \le v$ and $u' \le s' < t' \le v'$. If $u' < u$, then $u' < u \le s < s' < t$; with the same reasoning as above, we conclude that $\Delta Z(u) = \Delta Z(s') = 0$ and $Z(u) = Z(s') = Z(t) = \inf_{[u',t]} Z$ which contradicts \ref{H1}. Similarly, if $u' > u$, then $u < u' \le s' < t < t' \le v' < v$ and we conclude that $\Delta Z(u') = \Delta Z(t) = 0$ and $Z(u') = Z(t) = Z(t') = \inf_{[u,t']} Z$.

Next, we need to show that $L(Z,\boldsymbol{\ell})$ is closed. Consider a sequence of points of the plane $(x_n)$ on $L(Z,\boldsymbol{\ell})$ which converges as $n \to \infty$ to $x \in \overline{\D}$. Let us show that $x \in L(Z,\boldsymbol{\ell})$. If $x isn L(Z)$, then there exists a face $V$ of the latter such that $x \in V$ and, moreover, $x_n \in V$ for every $n$ large enough. Note that if $u$ is the jump time of $Z$ associated with $V$, then $V \cap L(Z,\boldsymbol{\ell}) = \bigcup_{t \in B_{u}} ]\e^{-2\i\pi p_{u}(\bell)}, \e^{-2\i\pi t}[$. Thus, for every $n$ large enough, $x_n$ belongs to a chord $[\e^{-2\i\pi p_{u}(\bell)}, \e^{-2\i\pi t_n}]$, where $t_n \in B_u$. Since $B_u$ is compact, upon extracting a subsequence, we may, and do, suppose that $t_n$ converges to a certain $t \in B_u$ as $n \to \infty$ and we conclude that $x \in [\e^{-2\i\pi p_{u}(\bell)}, \e^{-2\i\pi t}]$.

Finally, we show that $L(Z,\boldsymbol{\ell})$ is a maximal lamination. We argue by contradiction that for every $a, b \in \S$ with $a \ne b$, the open chord $(a, b) \coloneqq [a, b] \setminus\{a, b\}$ must intersect $L(Z,\boldsymbol{\ell})$, otherwise $L(Z,\boldsymbol{\ell}) \cup [a, b]$ would be a bigger lamination. Fix $0 \le s < t \le 1$ and suppose that $(\e^{-2\i\pi s}, \e^{-2\i\pi t}) \cap L(Z,\boldsymbol{\ell}) = \varnothing$. Then $(\e^{-2\i\pi s}, \e^{-2\i\pi t})$ belongs to a face $V_u$ for a certain $u \in J(Z)$. As a consequence $s,t \in B_u$,  so that, setting $v = \inf\{r > u : Z(r) = Z(u-)\}$, we have $s,t \in [u,v]$, $Z(s) = \inf_{[u,s]}Z$ and $Z(t) = \inf_{[u,t]}Z$. We claim that $Z(s) \ne Z(t)$ and so $Z(s) > Z(t)$. Indeed suppose $Z(s) = Z(t)$ and observe that $Z$ is continuous at $s$ by \ref{H3}; either $Z(r) > Z(s)$ for every $r \in (s, t)$ and then $[\e^{-2\i\pi s}, \e^{-2\i\pi t}] \subset L(Z)$, or there exists $r \in (s, t)$ such that $Z(s) = Z(r) = Z(t)$, which contradicts \ref{H1}. Let $x = \inf\{r > u : Z(r) \le (Z(s)+Z(t))/2\}$, then $x \in (s,t) \cap B_u$. Finally, note that $(s,t) \ne (u,v)$ so, similarly, there exists $y \in B_u \cap ((u,s) \cup (t,v))$. Since $p_u(\bell) \in B_u \setminus \{s,t\}$, we conclude that one of the open chords $(\e^{-2\i\pi p_u(\bell)}, \e^{-2\i\pi x})$ or $(\e^{-2\i\pi p_u(\bell)}, \e^{-2\i\pi y})$ intersects $(\e^{-2\i\pi s}, \e^{-2\i\pi t})$.
\end{proof}

As a consequence, note that $C_{u}(\boldsymbol{\ell})$ is compact for every $u \in L(Z)$.

\begin{rem}\label{rem:Lchapeau}For $\alpha \in (1,2)$, the triangulation $\widehat{\mathbf{L}}_{\alpha}$ introduced in \cite{Mar15} is a particular case of a triangulated lamination. Indeed, we have $\widehat{\mathbf{L}}_{\alpha}=L(\X,\boldsymbol{\ell}) $ with $\ell_{s}=0$ for every $s$. In other words, $\widehat{\mathbf{L}}_{\alpha}$ is obtained from the stable lamination $\mathbf{L}_{\alpha}$ by drawing chords from the ``leftmost'' vertex of a face to all the other vertices of this face.
\end{rem}

An interesting example  of a triangulated lamination is the so-called uniform $\alpha$-stable triangulation, which is defined as follows. For $\alpha \in (1,2)$, conditionally given $\X$, let  $\bell^{U}=(\ell_{s})_{s \in J(\X)}$ be a sequence of i.i.d.~uniform random variables on $[0,1]$. The uniform stable triangulation $ \mathbf{L}^{\rm U}_\alpha$ is then defined to be $ \mathbf{L}^{\rm U}_\alpha \coloneqq L(\X, \bell^{U})$. We will see that  $L(\X, \bell^{U})$ is the distributional limit of certain simply generated noncrossing trees as well as large critical Bienaymé--Galton--Watson trees in the domain of attraction of a stable law of index $\alpha$ which are uniformly embedded  in a noncrossing way.

\begin{rem}\label{rem:codage}
If $f : [0,1] \to \R_+$ is a continuous function such that $f(0)=f(1)=0$ but which does not fulfill \eqref{H_f}, one can still adapt the construction of $\mathbf{L}(f)$ is Section \ref{sec:trig} to define a (non-maximal) lamination from $f$, see Curien \& Le Gall \cite[Prop. 2.5]{CLGrecursive}. As shown in \cite{Kor11}, the stable laminations $\mathbf{L}_{\alpha}$ can be coded by $H^{\rm ex}_\alpha$,  the normalized excursion of the so-called height process associated with $\X$. In the same way, the sets $L(\X,\boldsymbol{\ell})$ could also be defined from $H^{\rm ex}_\alpha$ (although in a different sense than that of Curien \& Le Gall since it would involve $\bell$). Nonetheless, $H^{\rm ex}_\alpha$ is a more complicated object than $\X$, the definition of $p_u$ and the invariance principles of Section \ref{sec:inv} would be more technical and may even require more assumptions (see Remark \ref{rem:contour} below). 

Conversely, if $L$ is a maximal lamination, by adapting the argument of \cite[Prop.~2.2]{LGP08} and using \cite[Cor.~1.2]{Duq06}, we believe that there exists a continuous function $f : [0,1] \to \R_+$ with $ f(0)=f(1)=0$ satisfying \eqref{H_f} such that $L=\mathbf{L}(f)$. However, if $L$ is the lamination
$$L= \left[ 1,\e^{- \i \pi/2 } \right] \cup \left[ \e^{- \i \pi /2}, -1 \right] \cup \left[ -1 , \e^{ \i \pi/2 } \right] \cup \left[ \e^{ \i \pi /2 } , 1 \right] \cup \left[ -1 ,1 \right],$$
there does not exist a continuous  function $f : [0,1] \to \R_+$ with $ f(0)=f(1)=0$ such that $L=\mathbf{L}(f)$ in the sense of Curien \& Le Gall \cite[Prop. 2.5]{CLGrecursive}, and there does not exist a càdlàg function $Z \in \D([0,1],\R)$ satisfying \ref{H1},   \ref{H2},   \ref{H3},   \ref{H4} such that $L=L(Z)$.  
In the same way, $L(\X,\boldsymbol{\ell})$ cannot be coded by a continuous or a càdlàg function in this manner for $\alpha \in (1,2)$.
\end{rem}

\subsection{The Hausdorff dimension of  triangulated stable laminations}

If $L$ is a lamination, we let $A(L) \subset \S$ denote the set of all end-points of its chords. We denote by $\dim(K)$ the Hausdorff dimension of a subset $K$ of $\C$, and refer to Mattila \cite{Mat95} for background. Recall that $\X$ is the normalized excursion of the $\alpha$-stable Lévy process.

\begin{thm}\label{thm:dim_Hausdorff_triangulation_stable}
For every $\alpha \in (1,2)$ and for every jumps labelling $\bell$, almost surely, \begin{equation}\label{eq:dim_Hausdorff_triangulation_stable}
\dim(A(L(\X, \bell))) = \frac{1}{\alpha}
\qquad\text{and}\qquad
\dim(L(\X, \bell)) = 1 + \frac{1}{\alpha}.
\end{equation}
\end{thm}

These results should be compared with \cite[Thm.~5.1]{Kor11}, where these dimensions are calculated for stable laminations:
\begin{equation}\label{eq:dim_Hausdorff_lamination_stable}
\dim(A(L(\X))) = 1 - \frac{1}{\alpha}
\qquad\text{and}\qquad
\dim(L(\X)) = 2 - \frac{1}{\alpha}.
\end{equation}
We mention that \eqref{eq:dim_Hausdorff_lamination_stable} also holds for $\alpha=2$ by results of Aldous \cite{Ald94b} and Le Gall \& Paulin \cite{LGP08} when $L(X^{\rm ex}_2)$ is taken to be the Brownian triangulation.

We mention that Theorem \ref{thm:dim_Hausdorff_triangulation_stable} is established in \cite{Mar15} in the particular case where $\ell_{s}=0$ for every $s$. The general case only requires mild modifications, but we give a full proof for completeness.

\begin{rem}
We see that the dimensions of the sets in \eqref{eq:dim_Hausdorff_triangulation_stable} and \eqref{eq:dim_Hausdorff_lamination_stable} have the same limit as $\alpha \uparrow 2$. Indeed, the stable lamination and actually any triangulated stable lamination converge to the Brownian triangulation in this limit. On the other hand, we also see that
\begin{equation}
 \left( \dim(L(\X)), \dim(L(\X,\bell)) \right)   \quad \mathop{\longrightarrow}_{\alpha \downarrow 1} \quad  (1,2).
\end{equation}
Let us give an intuitive explanation of this fact. Informally, as $\alpha \downarrow 1$, the process $\X$ converges towards the deterministic function $f : [0,1] \to \R$ defined by $f(0) = 0$ and $f(x) = 1-x$ for every $x \in (0,1]$ ($f$ is not c{\`a}dl{\`a}g, but we refer to \cite[Theorem 3.6]{CK13} for a precise statement and proof). If we try then to define $L(f)$ and $L(f,\bell)$ mimicking \eqref{eq:lamination_stable} and \eqref{eq:triangulation_stable_canonique}, we obtain $L(f) = \S$ and $L(f,\bell) = \overline{\D}$.
\end{rem}

\begin{proof}[Proof of Theorem \ref{thm:dim_Hausdorff_triangulation_stable}]
Fix a face $V$ of $L(\X,\bell)$ and let $s$ be the jump-time of $\X$ associated with $V$. Notice from \eqref{eq:triangulation_stable_canonique} that all the chords of $L(\X, \bell)$ which lie in $\overline{V}$ either belong to the boundary $\partial V$ or are of the form $[\e^{-2\i\pi p_s(\bell)}, \e^{-2\i\pi r}]$ for $r \in \overline{V} \cap \S$. To simplify notation, denote by $L_V$ the lamination $L(\X, \bell) \cap \overline{V}$ and by $A_{V}$ the set of all its end-points, so that
\begin{equation}
A_V = \overline{V} \cap \S.
\end{equation}
and $\dim(A_V )=1/\alpha$ by \cite[Theorem 5.1]{Kor11}. As a consequence, since $A(L(\X, \bell)) = \bigcup_V A_V$, where the union runs over the countable set of faces of $L(\X)$, we have
\begin{equation}
\dim(A(L(\X, \bell))) = \sup_{V \text{ face of } L(\X)} \dim(A_V) = \dim(\overline{V} \cap \S) = \frac{1}{\alpha}.
\end{equation}
Similarly, we have
\begin{equation}
\dim(L(\X, \bell)) = \sup_{V \text{ face of } L(\X)} \dim(L_V)
\end{equation}
so it only remains to show that for any given face $V$ of $L(\X)$, we have
\begin{equation}\label{eq:Lv}
\dim(L_V) = 1 + \dim(A_V) = 1 + \frac{1}{\alpha}.
\end{equation}
If $s$ is the jump time associated with $V$, it is actually sufficient to establish \eqref{eq:Lv} with $L_{V}$ replaced by  the compact set $C_s(\bell)$, where we recall that $C_s(\bell)$ is the union of the chords $[\e^{-2\i\pi p_s(\bell)}, z]$ for $z \in A_V$. Indeed as we remarked previously, $L_{V} \setminus C_{s}(\bell) \subset L(\X)$ which, by \eqref{eq:dim_Hausdorff_lamination_stable}, has Hausdorff dimension $2-\frac{1}{\alpha} < 1+\frac{1}{\alpha}$ for every $\alpha\in (1,2)$. We adapt the argument of Le Gall \& Paulin \cite[Proposition 2.3]{LGP08} to show that $\dim(C_{s}(\bell)) =1+\dim(A_V)$.

We first show that $\dim(C_s(\bell)) \ge 1 + \dim(A_V)$. Fix $0 < \gamma < \dim(A_{V})$; thanks to Frostman's lemma \cite[Theorem 8.8]{Mat95}, there exists a non-trivial finite Borel measure $\nu$ supported on $A_{V}$ such that $
\nu(B(x,r)) \le r^\gamma$ for every $x \in \C$ and every $r > 0$, where $B(x,r)$ is the Euclidean ball centered at $x$ and of radius $r$. Next, for every $x \in A_{V}$, denote by $\lambda_x$ the one-dimensional Hausdorff measure on the chord joining $x$ to $\e^{-2\i\pi p_s(\bell)}$. We define a finite Borel measure $\Lambda$ on $\C$, supported on $C_s(\bell)$, by setting for every Borel set $B$
\begin{equation}
\Lambda(B) = \int \nu({\rm d}x) \lambda_x(B).
\end{equation}
Fix $0<R<1$ such that $\Lambda(B(0, R)) > 0$; let $z_0 \in B(0,R) \cap C_s(\bell)$ and then $x_0 \in A_{V}$ such that the chord $[x_0,\e^{-2\i\pi p_s(\bell)}]$ contains $z_0$. Fix $\varepsilon \in (0,1]$; every $x \in A_{V}$ such that the chord $[x,\e^{-2\i\pi p_s(\bell)}]$ intersects the ball $B(z_0, \varepsilon)$ must satisfy $|x-x_0| \le C \varepsilon$, where the constant $C$ only depends on $R$. We conclude that
\begin{equation}
\Lambda(B(z_0, \varepsilon))
= \int_{|x-x_0| \le C \varepsilon} \nu({\rm d}x) \lambda_x(B(z_0, \varepsilon))
\le C' \varepsilon^{1+\gamma},
\end{equation}
where the constant $C'$ does not depend on $\varepsilon$ nor $z_0$. Appealing again to Frostman's lemma, we obtain $\dim(C_s(\bell)) \ge 1 + \gamma$, whence, as $\gamma < \dim( A_{V})$ is arbitrary, $\dim( C_s(\bell)) \ge 1 + \dim( A_{V})$.

It remains to show the converse inequality. We denote respectively by $\dimMI(K)$ and $\dimMS(K)$ the lower and upper Minkowski dimensions of a subset $K$ of $\C$ (see e.g. Mattila \cite[Chapter 5]{Mat95}); recall that for every $K \subset \overline{\D}$, we have $\dim(K) \le \dimMI(K) \le \dimMS(K)$. Observe from the proof of Theorem 5.1 in \cite{Kor11} (in particular, Proposition 5.3 there) that we have $\dim(A_{V}) = \dimMS ( A_{V})$.
Fix $\beta > \dim(A_{V}) = \dimMI(A_{V})$; then there exists a sequence $(\varepsilon_k; k \ge 1)$ decreasing to $0$ such that for every $k \ge 1$, there exists a positive integer $M(\varepsilon_k) \le \varepsilon_k^{-\beta}$ and $M(\varepsilon_k)$ disjoint subarcs of $\S$ with length less than $\varepsilon_k$ and which cover $ A_{V}$.
It follows that the two-dimensional Lebesgue measure of the $\varepsilon_k$-enlargement of $C_s(\bell)$ is bounded above by $C \varepsilon_k^{1-\beta}$, where the constant $C$ does not depend on $k$. We conclude from \cite[page 79]{Mat95} that $\dim(C_s(\bell)) \le \dimMS(C_s(\bell)) \le 1 + \beta$ for every $\beta > \dim(A_{V})$, which completes the proof.
\end{proof}

\section{Invariance principle for triangulated laminations}
\label{sec:inv}

In this section, we establish invariance principles for different classes of noncrossing  trees which converge to  triangulated stable laminations. As an application, we obtain limit theorems for large discrete random trees embedded in a noncrossing way.

\subsection{The continuous case}

If $\tau$ is a plane tree, we let $\mathsf{H}(\tau)=\max_{u \in \tau} |u|$ be its height. Recall that $ \W(\tau)$ is its {\L}ukasiewicz path.

\begin{prop}\label{prop:invcontinu} Let $f:[0,1] \rightarrow \R_{+}$ be a continuous function satisfying  \eqref{H_f} and such that $f(0)=f(1)=0$. For every $ n \geq 1$, let $\theta_{n}$ be a noncrossing tree with $n$ vertices and let $\tau_{n}$ be its shape. Assume that, as $n \rightarrow \infty$,
\begin{enumerate}
\item[(i)] $\mathsf{H}(\tau_{n})/n \rightarrow 0$;
\item[(ii)] There exists a sequence $B_{n} \rightarrow \infty$ such that $ \W(\tau_{n})/B_{n} \rightarrow f$ for the uniform topology. 
\end{enumerate}
Then  the convergence $\theta_{n} \rightarrow \mathbf{L}(f)$ holds for the Hausdorff topology.
\end{prop}

In other words, as soon as the {\L}ukasiewicz path of the shape of a sequence of noncrossing trees converges to a continuous function having distinct local minima, the limit of the noncrossing trees is a triangulation that only depends on their shapes and not on their embeddings, provided that their height is negligible compared to their total size.

Also notice that Assumption (i) is crucial, as it simple to construct a sequence of noncrossing trees satisfying (ii) but which does not converge for the Hausdorff topology. In addition, note that we do not require the local minima of $f$ to be dense in Proposition \ref{prop:invcontinu}, so that $\mathbf{L}(f)$ may be a triangulation with nonempty interior.

\begin{cor}\label{cor:finitevar}Let $\mu$ be critical offspring distribution with finite variance. For every $n \geq 1$, let $\mathscr{T}_{n}$ be a random noncrossing tree with $n$ vertices such that its shape has the law $\mathrm{BGW}^\mu_n$. Then  $\mathscr{T}_{n}$ converges in distribution to the Brownian triangulation as $n \rightarrow \infty$.
\end{cor}

This result simply follows Proposition \ref{prop:invcontinu} by applying Skorokhod's representation theorem and combining \eqref{eq:Duq} with the well-known fact that $\mathsf{H}(S( \mathscr{T}_{n}))/\sqrt{n}$ converges in distribution to a positive random variable as $n \rightarrow \infty$.

\begin{rem}\label{rem:contour}
In \cite[Sec.~3.2]{CKdissections}, a similar result to Proposition \ref{prop:invcontinu} is established using the contour function with the additional assumptions that the leaves of $\tau_{n}$ are ``uniformly distributed'' and that the local minima of $f$ are dense. An important point is that we do not require the local minima of $f$ to be dense in Proposition \ref{prop:invcontinu}, which in particular allows triangulations with nonempty interior. We lift these restrictions by using the {\L}ukasiewicz path instead of the contour function. Another advantage of this approach is that invariance principles are usually simpler to establish for the {\L}ukasiewicz path than the contour function, and the fact that the leaves of $\tau_{n}$ are ``uniformly distributed'' does not necessarily follow from a functional invariance principle. For instance, Corollary \ref{cor:finitevar} applies to more general classes of random trees than Bienaymé--Galton--Watson trees, such as random trees with prescribed degree sequences \cite{BM14}.
\end{rem}

We start with a preliminary observation which will be crucial in the proof of Proposition \ref{prop:invcontinu}: roughly speaking, if the height of a plane tree is small compared to its size, then in any possible embedding of this plane tree as a noncrossing tree, the position of every vertex having a small number of descendants is known, up to a small error. In addition, if a vertex is such that only one of the subtrees grafted on its children is large, then it can only have two possible locations in the noncrossing embedding, up to a small error.

\begin{lem}\label{lem:nc}Let $\theta$ be a noncrossing with shape $\tau$ having $n$ vertices.  Denote by $\varnothing = u_{0} \prec u_{1} \prec \dots \prec u_{n-1}$ the vertices of $\tau$ labelled in lexicographical order.  
 Fix $\eta, \varepsilon \in (0,1)$. Let $0 \leq k \leq n-1$ and denote by $S_{k}$ the number of (strict) descendants of $u_{k}$. Assume that $\mathsf{H}(\tau)/n \leq \varepsilon$.
\begin{enumerate}
\item[(i)] Assume that $S_{k} \leq  \eta n$. Then
$$ \left| \e^{-2 \i \pi k/n }- u_{k} \right| \leq 7(\varepsilon+\eta),$$
where we identify $u_{k}$ with its associated complex number in the noncrossing tree $\theta$.
\item[(ii)] Let $M_{k}$ be the size of the largest subtree grafted on a child of $u_{k}$. Assume that $S_{k}-M_{k} \leq  \eta n$. Then
$$\min \left( \left | \e^{-2 \i \pi k/n }- u_{k} \right|,  \left | \e^{-2 \i \pi (k+S_{k})/n }- u_{k} \right| \right)  \leq 7( \varepsilon+\eta).$$
\end{enumerate}
\end{lem}

\begin{proof}Let $P_{k} \in \{0,1, \ldots,n-1\}$ be such that the vertex $u_{k}$ is the complex number $\exp(-2 \i \pi P_{k} /n)$ in $\theta_{n}$. Then $$ \left| k-P_{k} \right| \leq \mathsf{H}(\tau)+ S_{k}.$$
This readily follows by the description of the  bijection $\Phi_{n}^{-1}$ given in the proof of Proposition \ref{prop:bij}: the error $ \mathsf{H}(\tau)$ corresponds to the vertices belonging to $\llbracket \varnothing, u_{k} \llbracket$ which may be folded to the right of $u_{k}$ in $\theta$, and the error $S_{k}$ correspond to all the vertices after $u_{k}$ (in the lexicographical order) which may be folded to the left of $u_{k}$. Assertion (i) follows by using the fact that $|\e^{-2 \i \pi s}-\e^{-2 \i \pi t} | \leq 2 \pi |s-t|$ for $s,t \in [0,1]$.

For (ii), let $\widetilde{u}$ be a child of $u_{k}$ having $M_{k}$ descendants (including itself). Then either $\widetilde{u}$ is folded to the right of $u_{k}$ in $\theta$, in which case all these $M_{k}$ descendants are also folded to the right of $u_{k}$ in $\theta$, so that  $| k-P_{k} | \leq \mathsf{H}(\tau)+ S_{k}-M_{k}$, or $\widetilde{u}$ is folded to the left of $u_{k}$ in $\theta$, in which case all these $M_{k}$ descendants are also folded to the left of $u_{k}$ in $\theta$, so that $| k+M_{k}-P_{k} | \leq \mathsf{H}(\tau)+ S_{k}-M_{k}$  (the errors $S_{k}-M_{k}$ come from the  descendants of $u_{k}$ which are not descendants of $ \widetilde{u}$ and which may be folded to the left of $u_{k}$). This completes the proof.\end{proof}

\begin{proof}[Proof of Proposition \ref{prop:invcontinu}]  Since the space of compact subsets of $\Db$ equipped with the Hausdorff distance is compact and the space of laminations is closed, up to extraction we thus suppose that $(\theta_{n})_{n \geq 1}$ converges towards a lamination $L$ of $\Db$ and we aim at showing that $L=\mathbf{L}(f)$. Since $\mathbf{L}(f)$ is maximal, it suffices to check that $\mathbf{L}(f) \subset L$.

Fix $0<s<t<1$ such that $s \overset{f}{\thicksim} t$ and let us show that $[\e^{-2\i \pi s},\e^{-2\i \pi t}] \subset L$. To this end, we fix $\varepsilon \in (0, (t-s)/10)$ and show that $[\e^{-2\i \pi s},\e^{-2\i \pi t}] \subset \theta_{n}^{(49 \varepsilon)}$ for every $n$ sufficiently large, where $X ^{(\varepsilon)}$ is the $ \varepsilon$-enlargement of a closed subset $X \subset \Db$. Observe from \eqref{H_f} that either $f(s) = f(t) < f(r)$ for every $r \in (s,t)$, or there exists a unique $r \in (s,t)$ such that $f(s) = f(t) = f(r)$ and neither $s$ nor $t$ are times of a local minimum. We may restrict our attention to the first case since, in the second one, there exists $s' \in (s-\varepsilon, s)$ and $t' \in (t, t+\varepsilon)$ such that $f(s') = f(t') < f(r)$ for every $r \in (s',t')$. We assume in the sequel that $f(s) = f(t) < f(r)$ for every $r \in (s,t)$ and that $n$ is sufficiently large so that $\mathsf{H}(\tau_{n})/n \leq \varepsilon$.

We start with some preliminary observations. Let $\W^{(n)}$ be the {\L}ukasiewicz path of $\tau_{n}$ and denote by $\varnothing = u^{(n)}_{0} \prec u^{(n)}_{1} \prec \dots \prec u^{(n)}_{n-1}$ the vertices of $\tau_{n}$ labelled in lexicographical order. It is well known that $u^{(n)}_{i}$ is an ancestor of $u^{(n)}_{j}$ if and only if $i \leq j$ and $ \W^{(n)}_{i}= \min_{[i,j]} \W^{(n)}$ (see e.g.~\cite[Prop.~1.5]{LG05}). As a consequence, for every $0 \leq k \leq n-1$, if $S^{(n)}_{k}$ denotes the number of (strict) descendants of $u^{(n)}_{k}$, we have
\begin{equation}
|u^{(n)}_{k}|= \# \left\{ 0 \leq j \leq k-1 : \W^{(n)}_{j}  = \min_{[j,k]} \W^{(n)}  \right\}, \quad  S^{(n)}_{k}= \min \left\{j \geq k : \W^{(n)}_{j}< \W^{(n)}_{k} \right\}-k-1. 
\end{equation}

Since $f(r)>f(s)=f(t)$ for every $r \in (s,t)$, there exists $z \in (s,s+\varepsilon)$ such that $\inf \{u >z : f(u){{\le}}f(z)\}   \in   (t-\varepsilon,t)$.   As a consequence, setting $\eta=(z-s)/10$, for every $n$ sufficiently large, there exists $z_{n} \in  \{1, \dots, n - 1\}$ such that 
\begin{equation}
\label{eq:defZn} z- \eta \leq n^{-1} z_{n}  \leq  z+\eta, \qquad  t- 2 \varepsilon < n^{-1} \min \left\{i > z_{n} : \W^{(n)}_{i}{{\le}}\W^{(n)}_{z_{n}} \right\}<t.
\end{equation} 
 
Similarly, since $f(s) < \inf_{[z-4\eta,z+2\eta]}f$ , we can find $y_{n} \in \{1,\ldots,n-1\}$  such that $s \leq n^{-1} y_{n}  \leq   z-4\eta$,
 \begin{equation}
\label{eq:defYn}  t- 2 \varepsilon < n^{-1} \min \left\{i > y_{n} : \W^{(n)}_{i}{{\le}}\W^{(n)}_{y_{n}} \right\} \quad  \textrm{ and }   \quad n^{-1} \min \left\{i > y_{n} : \W^{(n)}_{i}<\W^{(n)}_{y_{n}} \right\}<t.
\end{equation}

We claim that for every $n$ sufficiently large there exists $r^{0}_n <  j^{0}_{n} \leq z_{n} \in \{1, \dots, n - 1\}$ such that 
\begin{equation}
\label{eq:rnjn}z- 3 \eta < n^{-1} r^{0}_{n}, \qquad z- 2 \eta  < n^{-1} j^{0}_{n}, \qquad \W^{(n)}_{r^{0}_{n}}>\W^{(n)}_{j^{0}_{n}}
\end{equation} 
Indeed, if this were not the case, for every $j \in ((z-2 \eta)n,z_{n})$, we would have $ \W^{(n)}_{r} \leq \W^{(n)}_{j}$ for every $r \in ((z-3\eta)n, j)$, yielding  $\W^{(n)}_{r}  = \min_{[r,z_{n}]} \W^{(n)}$ for every $(z-3\eta) n < r  < (z-2\eta)n$, which would imply that $|u ^{(n)}_{z_{n}}| \geq \eta n$ and contradict Assumption (i). 

 \begin{figure}[!h]
 \begin{center}
    \includegraphics[scale=1]{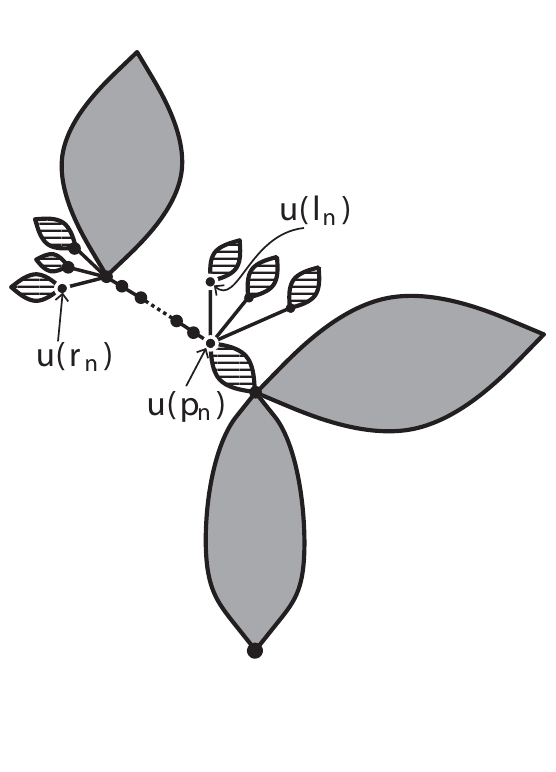}\quad \includegraphics[scale=1]{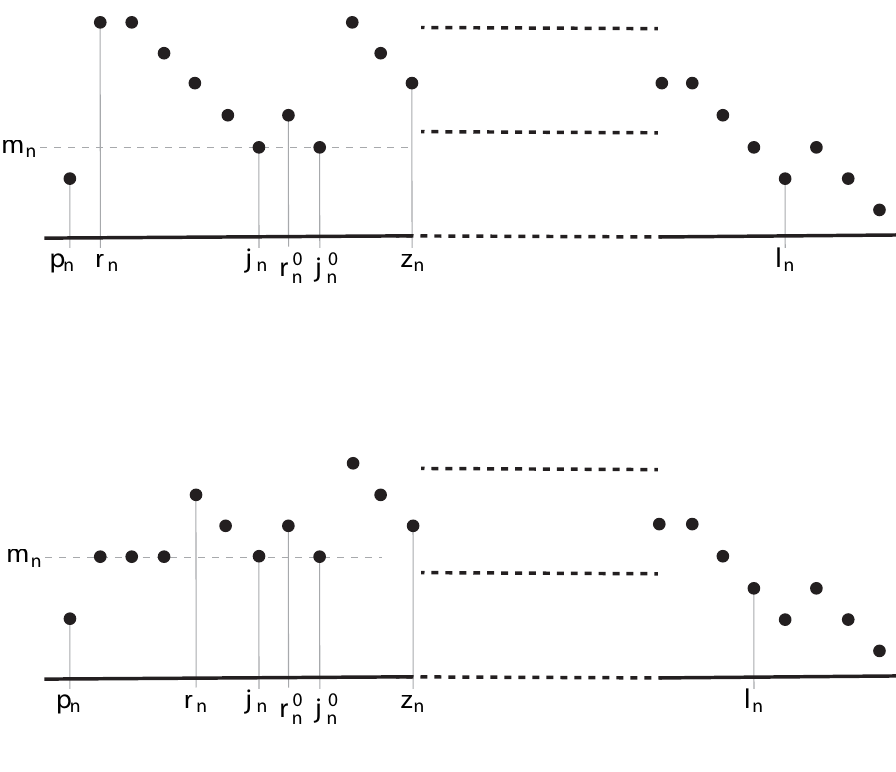}
   \caption{Illustration of the proof. On the left, the sizes of the dashed subtrees are small compared to the size of the  three grey subtrees. On the top right is illustrated the case where $ \W^{(n)}_{p_{n}+1}>m_{n}$ (so that $r_{n}=p_{n}+1$), and on the bottom right is illustrated the case where $\W^{(n)}_{p_{n}+1}=m_{n}$ (so that $r_{n}>p_{n}+1$). \label{fig:explic} }
 \end{center}
 \end{figure}

Choose $r^{0}_n <  j^{0}_{n} \leq z_{n} \in \{1, \dots, n - 1\}$ such that \eqref{eq:rnjn} holds.  Set
$$m_{n}= \min_{[r^{0}_{n},z_{n}]} \W^{(n)}, \qquad p_{n}= \max \left\{i< r^{0}_{n} : \W^{(n)}_{i} < m_{n}\right\}, \qquad r_{n}= \min \left \{i> p_{n}; \W^{(n)}_{i}>m_{n} \right\},$$
as well as
$$ j_{n}= \min \left\{i>r_{n} : \W^{(n)}_{i}=m_{n} \right\}, \quad     l_{n}= \min \left \{i>p_{n} :  \ \W^{(n)}_{i}<{{m_n}} \right\},$$
so that
$ y_{n} \leq p_{n} <r_{n}<j_{n} \le z_{n}<l_{n}$ and $\W^{(n)}_{i}=m_{n}$ for every $ p_{n}<i<r_{n}$. 
For the first inequality, note that $p_n < y_n$ would imply $\W^{(n)}_{y_n} \ge m_n$ and so $\min \{i > y_{n} : \W^{(n)}_{i} \le \W^{(n)}_{y_{n}}\} \le z_n$ which, by \eqref{eq:defZn}, contradicts \eqref{eq:defYn}. In addition, for every $n$ sufficiently large,
$$s \leq n^{-1} p_{n}, \qquad
n^{-1}j_{n} < s+2 \varepsilon,  \qquad n^{-1}j_{n}-n^{-1}p_{n} \leq 2 \varepsilon, \qquad  t-2\varepsilon<n^{-1} l_{n}<t.$$
The first inequality follows from the fact that $p_{n} \geq y_{n}$, the second one from the fact that 
$n^{-1}j_{n} \leq  n^{-1}z_{n} \leq z+\eta\leq s+2\varepsilon$, the third one from the first two, and the last one from \eqref{eq:defZn} and \eqref{eq:defYn}. Observe that $\W^{(n)}_{y_n} < \W^{(n)}_{p_n} < m_n$; we also have,
\begin{equation}
\label{eq:utile} n^{-1} \min \left \{i> p_{n} : \W^{(n)}_{i}<\W^{(n)}_{p_{n}} \right\}  < t
\end{equation}
by \eqref{eq:defYn}.

Note that either $\W^{(n)}_{p_{n}+1}>m_{n}$, in which case $r_{n}=p_{n}+1$ and $u^{(n)}_{r_{n}}$ is the first child of $u^{(n)}_{p_{n}}$, or $\W^{(n)}_{p_{n}+1}=m_{n}$ and so $u^{(n)}_{p_{n}+1}, \ldots, u^{(n)}_{r_{n}-2}$ all have one child, and $u^{(n)}_{r_{n}}$ is the first child of $u^{(n)}_{r_{n}-1}$.

This implies (see Fig.~\ref{fig:explic} for an illustration) that:
\begin{enumerate}
\item[(a)] $u^{(n)}_{l_{n}}$ is a child of $u^{(n)}_{p_{n}}$, since $\W^{(n)}_{p_n} \le \W^{(n)}_{l_n} < \W^{(n)}_{i}$ for every $p_n < i < l_n$;
\item[(b)] the number of descendants of $u^{(n)}_{r_{n}}$  is not greater than $2 \varepsilon n$, since, similarly, $\W^{(n)}_{j_n} < \W^{(n)}_{r_{n}}$ so that $S^{(n)}_{r_{n}} \leq r_{n}-j_{n} \leq r_{n}-p_{n} \leq 2 \varepsilon n$;
\item[(c)] the number of descendants of $u^{(n)}_{l_{n}}$ is not greater than $ 2\varepsilon n$ since
$$S^{(n)}_{l_{n}} \leq  \min \left\{i \geq l_{n}: \W_{i}^{(n)} < \W^{(n)}_{l_{n}} \right\}-l_{n} \leq nt-(t-2\varepsilon)n \leq 2\varepsilon n,$$
where we have used \eqref{eq:utile} for the second inequality.

\item[(d)] Fix $ p_{n} \leq i \leq r_{n}-1$. If $M^{(n)}_{i}$ denotes the size of the largest subtree grafted on a child of $u^{(n)}_{i}$, then $S^{(n)}_{i}-M^{(n)}_{i} \leq  4 \varepsilon n$. Indeed, note that this is trivial if $p_{n} < i < r_{n}-1$ since we observed that $u^{(n)}_{i}$ then has only one child; in the two other cases, we have $S^{(n)}_{i} \leq  nt-p_{n}$ using \eqref{eq:utile}, and in addition, $M^{(n)}_{i} \geq l_{n}-j_{n}$, so that
$$S^{(n)}_{i}-M^{(n)}_{i} \leq  (nt-l_{n})+(j_{n}-p_{n}) \leq 4 \varepsilon n.$$
\end{enumerate}

\emph{Step 1: Control of the positions of $u^{(n)}_{r_{n}}$ and $u^{(n)}_{l_{n}}$.} We claim that
\begin{equation}
\label{eq:control1}\left| \e^{-2 \i \pi s }- u^{(n)}_{r_{n}} \right| \leq 35 \varepsilon \quad\text{and}\quad \left| \e^{-2 \i \pi t }- u^{(n)}_{l_{n}} \right| \leq 35 \varepsilon.
\end{equation}
Indeed, By Lemma \ref{lem:nc} (i), we have $| \e^{-2 \i \pi r_{n}/n }- u^{(n)}_{r_{n}} | \leq 21\varepsilon$ by (b) and $| \e^{-2 \i \pi l_{n}/n }- u^{(n)}_{l_{n}} | \leq 21 \varepsilon$ by (c). Our claim then follows by the triangular inequality since $|\e^{-2 \i \pi r_{n}/n}- \e^{-2 \i \pi s}| \leq 7 |r_{n}/n-s| \leq 14 \varepsilon$ and $|\e^{-2 \i \pi l_{n}/n}- \e^{-2 \i \pi t}| \leq 7 |l_{n}/n-t| \leq 14 \varepsilon$.

\emph{Step 2: Control of the path between $u^{(n)}_{r_{n}}$ and $u^{(n)}_{p_{n}}$.} By (d), for every vertex $u^{(n)}_{k} \in \llbracket u^{(n)}_{p_{n}}, u^{(n)}_{r_{n}} \llbracket$ or, equivalently, for every $ p_{n} \leq k \leq r_{n}-1$, we have $S^{(n)}_{k}-M^{(n)}_{k} \leq  4 \varepsilon n$, so an application of Lemma \ref{lem:nc} (ii) yields
$$\min \left( \left | \e^{-2 \i \pi k/n }- u^{(n)}_{k} \right|,  \left | \e^{-2 \i \pi (k+S^{(n)}_{k})/n }- u^{(n)}_{k} \right| \right)  \leq 35 \varepsilon.$$
Note that $|\e^{-2 \i \pi k/n }-\e^{-2 \i \pi s}| \leq 7|k/n-s| \leq  7 (r_{n}/n-s) \leq  7(j_{n}/n-s) \leq 14 \varepsilon$. 
Also, $S_{k}^{(n)}+k \leq l_{n} < nt$, so that $ | \e^{-2 \i \pi (k+S^{(n)}_{k})/n }-\e^{-2 \i \pi t}| \leq 7 |(k+S^{(n)}_{k})/n-t| \leq 7(t-l_{n}) \leq 14 \varepsilon$. Therefore
\begin{equation}
\label{eq:control3}\min \left( \left | \e^{-2 \i \pi s}- u^{(n)}_{k} \right|,  \left | \e^{-2 \i \pi t/n }- u^{(n)}_{k} \right| \right)  \leq 49 \varepsilon.
\end{equation}

Since $u^{(n)}_{l_{n}}$ is a child of $u^{(n)}_{p_{n}}$ by (a), we conclude from \eqref{eq:control1} and \eqref{eq:control3} that   for every  $u \in \llbracket u^{(n)}_{r_{n}}, u^{(n)}_{l_{n}} \rrbracket$,
$$\min \left( \left | \e^{-2 \i \pi s}- u \right|,  \left | \e^{-2 \i \pi t/n }- u \right| \right)  \leq 49 \varepsilon.$$
Therefore, letting $ \mathcal{L}^{(n)}$ be the path  $\llbracket u^{(n)}_{r_{n}}, u^{(n)}_{l_{n}} \rrbracket$ in the noncrossing tree, we get that
$ \mathcal{L}^{(n)} \subset    [\e^{-2\i \pi s},\e^{-2\i \pi t}] ^{(49 \varepsilon)}$.
 Since $ \mathcal{L}^{(n)}$ is a union of finite segments joining $u^{(n)}_{r_{n}}$ to $u^{(n)}_{l_{n}}$, we get that 
$   [\e^{-2\i \pi s},\e^{-2\i \pi t}] \subset  (\mathcal{L}^{(n)})  ^{(49 \varepsilon)} \subset \theta_{n}^{(49 \varepsilon)}$, which establishes our original claim and completes the proof.
\end{proof}

\subsection{The càdlàg case}

Recall the definition of $p_{s}(\boldsymbol{\ell})$ from Sec.~\ref{sec:trilam}.

\begin{prop}\label{prop:invcadlag}Let $\theta_{n}$ be a noncrossing tree with $n$ vertices and shape $ \tau_{n}$. Denote by $\varnothing = u^{(n)}_{0} \prec u^{(n)}_{1} \prec \dots \prec u^{(n)}_{n-1}$ the vertices of $\tau_{n}$ listed in lexicographical order, let $k^{(n)}_{i}$ be the number of children of $u^{(n)}_{i}$ and let $L^{(n)}_{i}$ be the number of children of $u^{(n)}_{i}$ lying to the ``left'' of $u^{(n)}_{i}$ in $\theta_{{n}}$. Let $Z \in \D([0,1],\R)$ be a càdlàg function satisfying  \ref{H1}, \ref{H2}, \ref{H3}, \ref{H4}.
Assume that there exists a sequence $B_{n} \rightarrow \infty$ and a sequence $\bell = (\ell_{s}: s \in J(Z))$ indexed by the jump times of  $Z$  such that the following properties hold: \begin{enumerate}
\item[(i)] We have $\mathsf{H}(\tau_{n})/n \rightarrow 0$ as $n \rightarrow \infty$.
\item[(ii)] The convergence $\W(\tau_{n})/B_{n} \rightarrow Z$ holds for the Skorokhod topology.
\item[(iii)] For every $s \in (0,1)$, if $i_{n} \in \{0,1, \ldots,n-1\} $ is such that $\lim_{n \rightarrow \infty} k^{(n)}_{i_{n}}/B_{n}>0$ and $i_{n}/n \rightarrow s$, then $L^{(n)}_{i_{n}}/ k^{(n)}_{i_{n}} \rightarrow \ell_{s}$.
\item[(iv)] For every $s \in J(Z)$, $Z$ does not attain a local minimum at $p_{s}(\boldsymbol{\ell})$.
\end{enumerate}
 Then $\theta_{n} \rightarrow L(Z, \bell)$ for the Hausdorff topology.
\end{prop}

Roughly speaking, condition (iv) ensures that the special vertex from which each face is triangulated is not an endpoint of a chord of $L(Z,\bell)$ (but of course belongs to the closure of the endpoints of  chords).

 \begin{proof}
Since the space of compact subsets of $\overline{\D}$ equipped with the Hausdorff distance is compact and the space of laminations is closed, up to extraction we thus suppose that $(\theta_{n})_{n \geq 1}$ converges towards a lamination $L$ of $\overline{\D}$ and we aim at showing that $L=L(Z, \bell)$. Since $L(Z, \bell)$ is maximal, it suffices to check that $L(Z, \bell) \subset L$.

We first show that $L(Z) \subset L$. To this end, fix $\varepsilon > 0$ and choose $0 \le s < t \le 1$ such that $s \simeq^Z t$. If $\Delta Z(s)=0$, then $Z(t)=Z(s)= \inf_{[s,t]} Z$ and $\Delta Z(t)=0$. Arguments similar to those of the proof of Proposition \ref{prop:invcontinu} to show that $ [\e^{-2\i \pi s},\e^{-2\i \pi t}] \subset \theta_{n}^{(49 \varepsilon)}$ for $n$ sufficiently large. If $\Delta Z(s)>0$, then $t= \inf \{u>s : Z(t)=Z(s-)\}$ and for every $\varepsilon>0$ we have $\inf_{[s-\varepsilon,s]} Z<Z(s-)$ by \ref{H3} and $\inf_{[t,t+\varepsilon]} Z<Z(t)$ by \ref{H2}. Using these inequalities, again similar arguments to those of the proof of Proposition \ref{prop:invcontinu} yield that $ [\e^{-2\i \pi s},\e^{-2\i \pi t}] \subset \theta_{n}^{(49 \varepsilon)}$ for $n$ sufficiently large. We leave the (merely technical) details to the reader, and refer to \cite[Proof of Theorem 7.1]{Mar15} for detailed arguments.
 
Next, let $s \in J(Z)$, set $s'= \inf\{t > s : Z(t) = Z(s-)\}$ and fix $t \in [s,s']$ such that $Z(t) = \inf_{[s,t]}Z$ (observe that \ref{H3} implies $\Delta Z(t)=0$). We shall show that $[\e^{-2\i \pi p_{s}(\bell)},\e^{-2\i \pi t}] \subset \theta_{n}^{(\varepsilon)}$ for $n$ sufficiently large. Let $i_n$ as in (iii) and set
\begin{equation*}
S^{(n)}_{i_{n}}= \min \left\{j \geq i_{n}+1 : \W^{(n)}_{j}= \W^{(n)}_{i_{n}+1}-L^{(n)}_{i_{n}} \right\}-i_{n}-1,
\end{equation*}
the total number of (strict) descendants of the first $L^{(n)}_{i_{n}}$ children of $u^{(n)}_{i_{n}}$. Then, by definition of $L^{(n)}_{i_{n}}$,
\begin{equation*} \left| u^{(n)}_{i_{n}} - \e^{-2 \i \pi (i_{n}+S^{(n)}_{i_{n}})/n}\right| \leq 7 \frac{\mathsf{H}(\tau_{n})}{n},\end{equation*}
where the error term corresponds to the vertices belonging to $\llbracket \varnothing, u^{(n)}_{i_{n}} \llbracket$ which may be folded to the right of $u^{(n)}_{i_{n}}$ in $\theta_{n}$. Since $k^{(n)}_{i_{n}}/B_{n} \rightarrow \Delta Z(s)$, we have $L^{(n)}_{i_{n}}/B_{n} \rightarrow \Delta Z(s) \ell_{s}$. In addition, $\W^{(n)}_{i_{n}}/B_{n} \rightarrow Z(s)$. By (iv), $Z$ does not attain a local minimum at $p_{s}(\bell)$, so by continuity properties of first passage times for the Skorokhod topology,
\begin{equation*}
n^{-1} \cdot \min \left\{j \geq i_{n} :  \W^{(n)}_{j}= \W^{(n)}_{i_{n}+1}-L^{(n)}_{i_{n}} \right\}  \quad \mathop{\longrightarrow}_{n \rightarrow \infty} \quad \inf \left\{t \geq s : Z_{t}=Z_{s} - \Delta Z (s) \cdot \ell_{s} \right\}=p_{s}(\bell).
\end{equation*}
Therefore $n^{-1}(S^{(n)}_{i_{n}}+i_{n}) \rightarrow p_{s}(\bell) $, implying, by the previous bound and (i) that 
\begin{equation}\label{eq:truc1}
\left| \e^{-2\i\pi p_{s}(\boldsymbol{\ell})} - u^{(n)}_{i_{n}} \right| \cv 0.
\end{equation}

\begin{figure}[!h]
 \begin{center}
    \includegraphics[scale=1.1]{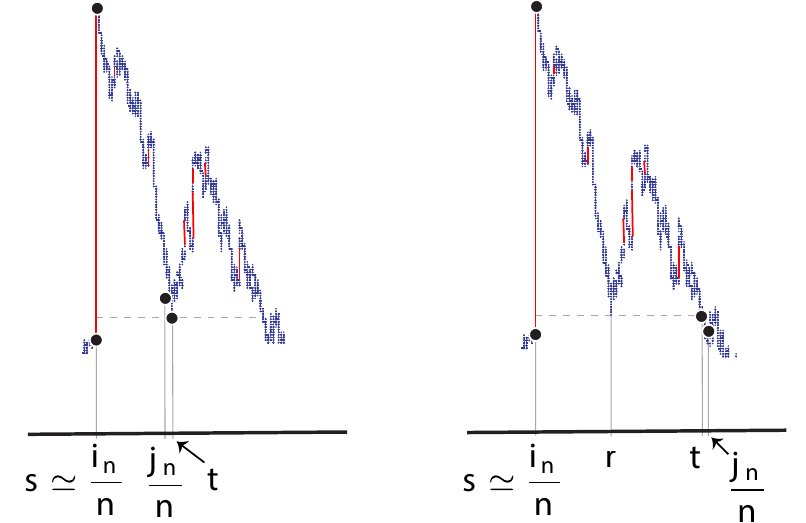}
    \caption{Illustration of the choice of $j_{n}$. On the left, the case where $Z(r)>Z(t)$ for every $r \in (s,t]$ and on the right, the case where there exists (a unique) $r \in (s,t)$ such that $Z(r)=Z(t)$.\label{fig:cas1et2} }
 \end{center}
\end{figure}

We claim that there exists $j_{n} \in \{0,1, \ldots,n-1\} $ such that $j_{n}/n \rightarrow t$, $u^{(n)}_{j_{n}}$ is a child of $u^{(n)}_{i_{n}}$ and the number of descendants of $u^{(n)}_{j_{n}}$ is $o(n)$ as $n \rightarrow \infty$. Indeed, suppose first that $Z(r)>Z(t)$ for every $r \in (s,t)$. Fix $\varepsilon \in (0,t-s)$; from \ref{H3}, the infimum of $Z$ over $[s,t-\varepsilon]$ is achieved at some point of this interval. Therefore, for $n$ large enough, there exists an integer $j_{n}$ such that $j_n/n \in [t-\varepsilon,t]$, $\W_{m}>\W_{j_{n}}$ for every integer $m \in [i_n+1, j_n-1]$, and $\inf\{l > j_n : \W_l = W_{j_n}-1\} \le j_n + n \varepsilon$; the claim then follows. Suppose next that there exists $r \in (s,t)$ such that $Z(r)=Z(t)=\inf_{[s,t]}Z$; then note that $r$ must be a time of local minimum by \ref{H3}, so this can only occur when $Z(t) > Z(s-)$ because otherwise it would contradict \ref{H4}, also $t$ cannot be a time of a local minimum by \ref{H1}. We conclude that for every $\varepsilon > 0$, we can find $t' \in (t, t+\varepsilon)$ such that $Z(s)<Z(t')<Z(r)$ for every $r \in (s,t')$ and the previous approximation thus applies.

This implies that $| \e^{-2\i \pi j_{n}/n} - u^{(n)}_{j_{n}}| \to 0$ by Lemma \ref{lem:nc} (i), so that 
\begin{equation}\label{eq:truc2}
\left| \e^{-2\i \pi t} - u^{(n)}_{j_{n}} \right| \cv 0.
\end{equation}
Combining \eqref{eq:truc1} and \eqref{eq:truc2}, since $u^{(n)}_{j_{n}}$ is a child of $u^{(n)}_{i_{n}}$, we get that for every $n$ sufficiently large
\begin{equation*}
\left[\e^{-2\i\pi p_{s}(\boldsymbol{\ell})}, \e^{-2\i\pi t} \right] \subset \left[u^{(n)}_{i_{n}},u^{(n)}_{j_{n}} \right]^{(\varepsilon)} \subset \theta_{n}^{(\varepsilon)}.
\end{equation*}
This completes the proof.
\end{proof}

\subsection{The uniform stable triangulation}

If $ \tau$ is a plane tree, we  set $\Theta^{U}(\tau)=\Phi_{n}^{-1}(\tau, \mathcal{C})$, where $ \mathcal{C}$ is a random element of $ \mathbb{C}(\tau)$ chosen uniformly at random.  In other words, $\Theta^{U}(\tau)$  is a noncrossing tree obtained by a ``uniform'' embedding of $\tau$.

Our next result establishes an invariance principle for large critical Bienaymé--Galton--Watson trees in the domain of attraction of a stable law of index $\alpha \in (1,2)$ which are embedded uniformly in a noncrossing way. The distributional limit is the uniform stable triangulation, which was introduced in Sec.~\ref{sec:trilam}.

\begin{thm}\label{thm:convergence_plongement_uniforme}
Fix $\alpha \in (1,2)$. For every critical offspring distribution $\mu$ belonging to the domain of attraction of a stable law of index $\alpha$, if $\CRT_n$ is a {\GW} tree with offspring distribution $\mu$ conditioned to have $n$ vertices, the convergence
\begin{equation}
\Theta^{\rm U}(\CRT_n) \cvloi \mathbf{L}^{\rm U}_\alpha
\end{equation}
holds in distribution for the Hausdorff distance on the space of all compact subsets of $\overline{\D}$.
\end{thm}

\begin{proof}
We want to apply Skorokhod's representation theorem and Proposition \ref{prop:invcadlag} with $Z = \X$. Assumptions (i) and (ii) hold by \eqref{eq:Duq} as well as the fact $\frac{B_{n}}{n} \mathsf{H}(S( \mathscr{T}_{n}))$ converges in distribution to a positive random variable as $n \rightarrow \infty$ \cite{Du03}. To see that Assumption (iii) holds, denote by $\varnothing = u^{(n)}_{0} \prec u^{(n)}_{1} \prec \dots \prec u^{(n)}_{n-1}$ the vertices of $\CRT_n$ listed in lexicographical order, let $k^{(n)}_{i}$ be the number of children of $u^{(n)}_{i}$ and let $L^{(n)}_{i}$ be the number of children of $u^{(n)}_{i}$ lying to the ``left'' of $u^{(n)}_{i}$ in $\Theta^{\rm U}(\CRT_n)$.  By definition, conditionally given $ \CRT_{n}$, $L^{(n)}_{i}$ is uniform on $ \{0,1, \ldots,k^{(n)}_{i}\}$, and the random variables $(L^{(n)}_{i} : 0 \leq i \leq n-1)$ are independent. In particular, conditionally on $k^{(n)}_{i_{n}} \rightarrow \infty$, $L^{(n)}_{i_{n}}/ k^{(n)}_{i_{n}}$ converges in distribution to a uniform random variable on $[0,1]$.
Finally, Assumption (iv) holds: almost surely, for every $s \in J(\X)$, $\X$ does not attain a local minimum at $p_{s}(\bell^U)$, where, conditionally given $\X$, $\bell^{U}=(\ell_{s})_{s \in J(\X)}$ is a sequence of i.i.d. uniform random variables on $[0,1]$. Indeed, almost surely, the times at which $\X$ attains a local minimum are at most countable, so for every $s \in J(\X)$, the probability that $p_{s}(\bell^U)$ is such a time is zero and, almost surely, $J(\X)$ is countable.
\end{proof}

\section{Applications to simply generated noncrossing trees}
In this section, we consider simply generated noncrossing trees, as defined by \eqref{eq:def_arbres_nc_simplement_generes}. We first prove that such trees are almost Bienaymé--Galton--Watson trees, and then establish Theorem \ref{thm:convergence_arbres_SG} by using the invariance principles obtained in the previous section.

We denote by $\mathrm{BGW}^{\mu_\varnothing, \mu}$ the law of a modified {\GW} tree, where the offspring distribution of the root is $\mu_\varnothing$, and that of the other vertices is $\mu$. For every integer $n$, we denote by $\mathrm{BGW}_n^{\mu_\varnothing, \mu}$ the law of such a tree conditioned to have $n$ vertices.

\subsection{Simply generated noncrossing trees are almost {\GW} trees}
As we have seen, every noncrossing tree  $\theta$ carries a planar structure, canonically rooted at the vertex corresponding to the complex number $1$, which is called the shape of $\theta$ and is denoted by $S(\theta)$. If $\mathscr{T}_n$ a random noncrossing tree uniformly distributed on $\NC_n$, then Thm. 1 in \cite{MP02} shows that $S(\mathscr{T}_{n})$ is a modified {\GW} tree, where the root has a different offspring distribution, conditioned to have size $n$. Our next result extends this to simply generated noncrossing trees.

\begin{thm}\label{thm:arbre_SG_presque_GW}
Assume that
\begin{equation}
\rho \quad \coloneqq \quad  \left(\limsup_{k \to \infty} w(k)^{1/k}\right)^{-1} >0.
\end{equation}
Fix $b \in (0,\rho)$,  set
\begin{equation}
a = \left(\sum_{k=0}^\infty (k+1) w(k+1) b^k\right)^{-1}
\qquad\text{and}\qquad
c = \left(\sum_{k=1}^\infty w(k) b^k\right)^{-1},
\end{equation}
and define
\begin{equation}\label{eq:lois_de_reproduction_cas_general}
\begin{cases}   
\mu(k) = a (k+1) w(k+1) b^k & (k \ge 0),
\\
\mu_\varnothing(k) = c w(k) b^k & (k \ge 1).
\end{cases}
\end{equation}
Then the law of the shape of a noncrossing tree sampled according to $\P_n^w$ is $\mathrm{BGW}_n^{\mu_\varnothing, \mu}$.
\end{thm}

Observe that
\begin{equation}\label{eq:loi_racine_moyenne_finie}
\sum_{j=1}^\infty j \mu_\varnothing(j) = \frac{bc}{a},
\qquad\text{whence}\qquad
\frac{k \mu_\varnothing(k)}{\sum_{j=1}^\infty j \mu_\varnothing(j)} = \mu(k-1).
\end{equation}
We shall see that the probability that the root of a modified {\GW} tree conditioned to have $n$ vertices has $k$ children converges towards $\frac{k \mu_\varnothing(k)}{\sum_{j=1}^\infty j \mu_\varnothing(j)}$ as $n \to \infty$. The above identity then translates roughly the fact that in a large modified {\GW} tree as above, the law of the degree of the root is close to that of the other vertices, as it is the case for a simply generated noncrossing tree.

\begin{rem}\label{rem:b}The condition \eqref{eq:defb} appearing in Theorem \ref{thm:convergence_arbres_SG} is equivalent to the fact that the probability measure $\mu$ defined by \eqref{eq:lois_de_reproduction_cas_general} can be chosen to be critical; in this case, it is unique. Indeed, consider the function
\begin{equation}\label{eq:Psi}
\Psi : x \in [0,\rho) \quad \longmapsto \quad  \frac{\sum_{k=0}^\infty k (k+1) w(k+1) x^k}{\sum_{k=0}^\infty (k+1) w(k+1) x^k}.
\end{equation}
Janson \cite[Lem. 3.1]{Jan12} observed that $\Psi$ is null at $0$, continuous and  increasing. Therefore, for every value $m \in (0, \Psi(\rho))$, where $\Psi(\rho) \coloneqq \lim_{x \uparrow \rho} \Psi(x)$, there exists a unique probability measure $\mu$ of the form \eqref{eq:lois_de_reproduction_cas_general} with expectation $m$. In particular, one can choose $\mu$ to be critical if and only if
\begin{equation}
\lim_{x \uparrow \rho} \Psi(x) \ge 1,
\end{equation}
in which case, $b > 0$ is the unique number such that
\begin{equation}\label{eq:hypothese_thm_arbre_SG_presque_GW}
\sum_{k=0}^\infty  (k+1)(k-1)w(k+1) b^k = 0.
\end{equation}
\end{rem}

\begin{rem}\label{rem:Marckert_Panholzer}
Consider the uniform distribution on noncrossing trees: $w(k) = 1$ for every $k \ge 1$. Then \eqref{eq:hypothese_thm_arbre_SG_presque_GW} holds with $b=1/3$. A simple calculation yields $a=4/9$ and $c=2$, so that \eqref{eq:lois_de_reproduction_cas_general} reads
\begin{equation}\label{eq:lois_de_reproduction_cas_uniforme}
\begin {cases}   
\mu(k) = 4 (k+1) 3^{-(k+2)} & (k \ge 0),
\\
\mu_\varnothing(k) = 2 \times 3^{-k} & (k \ge 1).
\end {cases}
\end{equation}
In particular, Thm. \ref{thm:arbre_SG_presque_GW} recovers the special case of Marckert \& Panholzer \cite[Thm. 1]{MP02}.
\end{rem}

\begin{proof}[Proof of Theorem \ref{thm:arbre_SG_presque_GW}]
Fix $n \ge 1$ and denote by $\Q_n^w$ the law of the shape of a random noncrossing tree sampled according to $\P_n^w$. We aim at showing that $\Q_n^w = \mathrm{BGW}_n^{\mu_\varnothing, \mu}$. To this end, fix $\tau \in \T_{n}$, and let $k_{0}, k_{1}, \ldots, k_{n-1}$ be the number of children of its vertices listed in lexicographical order (in particular, $k_{0}$ is the number of children of its root). By definition,
$$\mathrm{BGW}^{\mu_\varnothing, \mu}(\tau)=\mu_\varnothing(k_0) \prod_{i=1}^{n-1} \mu(k_i)=c w(k_0) b^{k_0} \prod_{i=1}^{n-1} a (k_i+1) w(k_i+1) b^{k_i}.$$
Note that $\sum_{i=0}^{n-1} k_{i}=n-1$, whence
$$\mathrm{BGW}^{\mu_\varnothing, \mu}(\tau)
=  c a^{n-1} b^{n-1} w(k_0) \prod_{i=1}^{n-1} (k_i+1) w(k_i+1).$$

Next, observe that $\P_n^w(\theta)$ only depends on the shape of $\theta$ and that $\# \{\theta \in \NC_{n} : S(\theta)=\tau\}=\#\mathbb{C}(\tau) = \prod_{i=1}^{n-1}(k_{i}+1)$ by Proposition \ref{prop:bij}. It follows that
$$\Q_n^w(\tau)= \sum_{\theta \in \NC_{n} : S(\theta)=\tau}  \P_n^w(\theta)= \frac{1}{Z_{n}^{w}}\#\mathbb{C}(\tau) \cdot \prod_{u \in \tau} w(\deg u)=\frac{1}{Z_n^w} w(k_0) \prod_{i=1}^{n-1} (k_i+1) w(k_i+1).$$
Since $\Q_n^w$ and ${\rm BGW}^{\mu_\varnothing, \mu}_n$ are both probability measures on $\T_n$, we conclude that  we have the identity $c a^{n-1} b^{n-1}/\mathrm{BGW}^{\mu_\varnothing, \mu}(\T_{n}) = 1/Z_n^w$ and the claim follows.
\end{proof}

\subsection{Largest subtree of the root of large modified {\GW} trees}
\label{sec:subtree}

Finally, Theorem \ref{thm:convergence_arbres_SG} will readily follow from the proof of Theorem \ref{thm:convergence_plongement_uniforme} and the next convergence, which extends Duquesne's theorem \eqref{eq:Duq} to modified Bienaymé--Galton--Watson trees.

\begin{thm}\label{thm:convergence_contour_arbre_SG}
Fix $\alpha \in (1,2]$. Let $\mu_\varnothing$ be a probability measure on $\N$ with finite mean and $\mu$ a probability measure on $\Z_+$ which is critical and belongs to the domain of attraction of a stable law with index $\alpha$. For every integer $n \ge 1$, sample $\CRT_n$ according to ${\rm BGW}^{\mu_\varnothing, \mu}_n$ (provided that ${\rm BGW}^{\mu_\varnothing, \mu}_n$  is well defined). Then \begin{equation}
\left(\frac{1}{B_n} \W_{\lfloor ns \rfloor}(\CRT_n) : s \in [0,1]\right)  \quad \mathop{\longrightarrow}_{n \rightarrow \infty} \quad  (\X(s) : s \in [0,1]),
\end{equation}
where the convergence holds in distribution in the space $\D([0, 1], \R)$ and where $(B_{n})_{n \geq 1}$ is the same sequence as in \eqref{eq:Duq}. \end{thm}

Marckert \& Panholzer \cite{MP02} obtained this limit theorem in the case where $\mu_\varnothing$ and $\mu$ are given by \eqref{eq:lois_de_reproduction_cas_uniforme}. We follow the same approach in the general case, which roughly speaking consists in comparing ${\rm BGW}^{\mu_\varnothing, \mu}_n$ and ${\rm BGW}^\mu_n$. However, Marckert \& Panholzer crucially use the fact that the support of $\mu$ and that of $\mu_\varnothing$ differ only at $0$. This is not the case when $\mu_\varnothing$ and $\mu$ are given by \eqref{eq:lois_de_reproduction_cas_general} as soon as $w(k) = 0$ for some $k \ge 1$, so some care is needed (see Remark \ref{rem:diff}). Our approach also gives a limit theorem for the size of the maximal subtree grafted on the root of a size-conditioned (possibly modified) Bienaymé--Galton--Watson tree.

We start by proving Theorem \ref{thm:convergence_arbres_SG}, assuming that Theorem \ref{thm:convergence_contour_arbre_SG} holds.

\begin{proof}[Proof of Theorem \ref{thm:convergence_arbres_SG}]
Define $\mu$ and $\mu_\varnothing$ by \eqref{eq:lois_de_reproduction_cas_general}, so that  the shape of $\mathscr{T}_n$ has law ${\rm BGW}^{\mu_\varnothing, \mu}_n$ by Theorem \ref{thm:arbre_SG_presque_GW}. In addition, the proof of Theorem \ref{thm:arbre_SG_presque_GW} also shows that conditionally given the shape $S(\mathscr{T}_n)$, the random variable $C( \mathscr{T}_{n})$ is uniformly distributed on the set of all its possible values. Under the assumption of Theorem \ref{thm:convergence_arbres_SG}, $\mu$ is critical and in the domain of attraction of a stable law of index $\alpha$. Since $\mu_\varnothing$ has finite mean by \eqref{eq:loi_racine_moyenne_finie}, we can apply Theorem \ref{thm:convergence_contour_arbre_SG}  and  conclude as in the proof of Theorem \ref{thm:convergence_plongement_uniforme}.
\end{proof}

\begin{rem}\label{rem:weights}If $k \mapsto w(k+1)$ is a critical probability distribution on $\Z_{+}$ belonging to the domain of attraction of a stable law of index $\alpha \in (1,2)$, a simply generated noncrossing tree with weights $w$ will converge to the Brownian triangulation (and its shape to the Brownian CRT), but a simply generated plane tree with weights $w$ will converge, appropriately rescaled, to the $\alpha$-stable random tree, and embedded in a uniform manner it will converge to the uniform $\alpha$-stable triangulation.
\end{rem}

We fix for the following $\mu_\varnothing$ a probability measure on $\N$ with finite mean and $\mu$ a probability measure on $\Z_+$ which is critical and belongs to the domain of attraction of a stable law with index $\alpha \in (1,2]$. We further assume that $\mu$ is aperiodic to avoid unnecessary complications, meaning that $ \gcd \{i>0 : \mu(i)>0\}=1$ so that ${\rm BGW}^{ \mu}( |\mathcal{T}|=n)>0 $ for every $n$ sufficiently large. The key estimate is the following, which may be of independent interest.

\begin{prop}\label{prop:plus_grand_sous_arbre_GW_modifie}
Denote by $M(\tau)$ the size of the largest subtree of the root of a plane tree $\tau$. Let $ \mathcal{N}$ be a random variable with law given by
$$\Pr{ \mathcal{N}=k } = \frac{k \mu_\varnothing(k)}{\sum_{j \ge 1} j \mu_\varnothing(j)}
\qquad (k \geq 1)$$ and let $(Y_{i})_{i \geq 1}$ be an independent sequence of i.i.d.~random variables having the law of the total size of a ${\rm BGW}^{\mu}$ tree.  Then, for every $k \geq 0$ and $L \geq 1$,
\begin{equation}
{\rm BGW}^{\mu_\varnothing, \mu}_n(n-1-M  = k,N_{0}=L)  \quad \mathop{\longrightarrow}_{n \rightarrow \infty} \quad \Pr{Y_{1}+Y_{2}+ \cdots+Y_{ L-1}=k, \mathcal{N}=L }.
\end{equation}
\end{prop}

Note that this implies that for every $k \geq 0$,  $${\rm BGW}^{\mu_\varnothing, \mu}_n(n-1-M  = k)  \quad \mathop{\longrightarrow}_{n \rightarrow \infty} \quad \Pr{Y_{1}+Y_{2}+ \cdots+Y_{ \mathcal{N}-1}=k}.$$
In particular,
{under ${\rm BGW}^{\mu_\varnothing, \mu}_n$, $M/n \to 1$ in probability as $n \to \infty$}, which was proved by Marckert \& Panholzer when  $\mu_\varnothing$ and $\mu$ are given by \eqref{eq:lois_de_reproduction_cas_uniforme}. Note also that this result covers the case of Bienaymé--Galton--Watson trees by taking $\mu_{\varnothing}=\mu$.

We establish Proposition \ref{prop:plus_grand_sous_arbre_GW_modifie} in several steps and first introduce some notation. Let $S=(S_{n})_{n \geq 0}$ be the random walk started from $0$ with step distribution $(\mu(k+1) : k \ge -1)$. Observe that $S$ is an aperiodic centered random walk with step distribution in the domain of attraction of a stable law with index $\alpha$. Recall the spectrally positive L{\'e}vy process $X_\alpha$ introduced in Sec.~\ref{sec:planetrees} and denote by $p_1$ the density of $X_\alpha(1)$; the latter is known to be positive, continuous and bounded (see e.g. Zolotarev \cite[I. 4]{Zol86}). We will use the local limit theorem (see Ibragimov \& Linnik \cite[Theorem 4.2.1]{IL71}), which tells us that 
\begin{equation}\label{eq:thm_limite_locale}
\sup_{k \in \Z} \big|B_n \P(S_n = k) - p_1(B_n^{-1} k)\big| \cv 0.
\end{equation}
For every $k \geq 1$, denote by $T_{-k}$ the first hitting time of $-k$ by the random walk $(S_{n})_{n \geq 0}$. We will need Kemperman’s formula, which states that 
\begin{equation}\label{eq:Kemperman} \Pr{T_{-k}=n}= \frac{k}{n} \cdot \Pr{S_{n}=-k}
\end{equation}
for every $k \geq 1$ and $n \geq 1$ (see e.g.  \cite[Chap. 6]{Pit06}). In particular,  the total size $Y_{1}$ of a ${\rm BGW}^{\mu}$ tree  belongs to the domain of attraction of a stable law of index $1/\alpha$, since $\Pr{Y_{1}=n} = \Pr{T_{-1}=n}= \frac{1}{n} \Pr{S_{n}=-1} \sim (nB_{n})^{-1} p_{1}(0)$ as $n \to \infty$.

The main tool to prove prove Proposition \ref{prop:plus_grand_sous_arbre_GW_modifie} is the following Lemma.

\begin{lem}\label{lem:maintool}~
\begin{enumerate}
\item[(i)] We have
$${\rm BGW}^{\mu_\varnothing, \mu}( |\mathcal{T}| =n)  \quad \mathop{\sim}_{n \rightarrow \infty} \quad |\Gamma(-1/\alpha)|^{-1} \cdot \left(\sum_{k \ge 1} k \mu_\varnothing(k) \right)  \cdot \frac{1}{n \cdot B_{n}}.$$
\item[(ii)] Denote by $N_0(\tau)$ the number of children of the root of a plane tree $\tau$. We have
\begin{equation}\label{eq:convergence_degre_racine}
{\rm BGW}^{\mu_\varnothing, \mu}_n(N_0=k) \cv \frac{k \mu_\varnothing(k)}{\sum_{j \ge 1} j \mu_\varnothing(j)}
\qquad\text{uniformly in }k.
\end{equation}
\item[(iii)] Fix $k \ge 1$; for every $n \ge k$, consider a forest of $k$ independent {\GW} trees with offspring distribution $\mu$, conditioned to have total size $n$ and denote by $M_{n}^{\mu,k}$ the size of the largest tree. Then, as $n \rightarrow \infty$, $n-M_{n}^{\mu,k}$ converges in distribution to the total size of $k-1$ independent {\GW} trees with offspring distribution $\mu$.
\end{enumerate}
\end{lem}

In particular, with the notation of (iii), $\frac{1}{n} M_{n}^{\mu,k} \rightarrow 1$ in probability as $n \rightarrow \infty$.

\begin{proof}
Observe that under ${\rm BGW}^{\mu_\varnothing, \mu}$, the {\L}ukasiewicz path associated with the tree is distributed as a random walk issued from $0$, with first step distributed as $(\mu_\varnothing(k+1) : k \ge 0)$ and the next ones as $(\mu(k+1) : k \ge -1)$, stopped at its first hitting time of $-1$. As a consequence, by decomposing the {\L}ukasiewicz path after the first step, for every $k \ge 1$ we have:
\begin{equation}\label{eq:size}
{\rm BGW}^{\mu_\varnothing, \mu}( |\mathcal{T}| =n)= \sum_{k=1}^{n-1} \mu_{\varnothing}(k) \cdot \Pr{T_{-k} =n-1}=  \sum_{k=1}^{n-1} \mu_{\varnothing}(k) \cdot \frac{k}{n-1} \Pr{S_{n-1}=-k},
\end{equation}
where  we have used Kemperman's formula for the last equality. Next note that for every fixed $k \geq 1$, we have
\begin{equation}
k \mu_\varnothing(k) B_{n-1} \P(S_{n-1} = -k) = k \mu_\varnothing(k)  \left( p_1(-B_{n-1}^{-1} k) + o(1) \right)  \cv k \mu_\varnothing(k) p_1(0),
\end{equation}
where the $o(1)$ is uniform in $k$. Since $\sum_{k \ge 1} k \mu_\varnothing(k) < \infty$ and $p_1$ is bounded, the above convergence yields also
\begin{equation}\label{eq:estimee}
\sum_{k \ge 1} k \mu_\varnothing(k) B_{n-1} \P(S_{n-1} = -k) \cv \sum_{k \ge 1} k \mu_\varnothing(k) p_1(0).
\end{equation}
Then (i) follows since $p_{1}(0)=|\Gamma(-1/\alpha)|^{-1}$ (see \cite[Lemma XVII.6.1]{Fel71}) and the fact that $(B_n)$ is regularly varying with index $1/\alpha$ which implies that $B_{n-1}/B_{n} \rightarrow 1$ as $n \rightarrow \infty$.

We now establish (ii). As in the proof of (i), also using \eqref{eq:size}, we have
\begin{equation}
{\rm BGW}^{\mu_\varnothing, \mu}_n(N_0=k)
= \frac{\mu_\varnothing(k) \P(T_{-k} = n-1)}{{\rm BGW}^{\mu_\varnothing, \mu}( |\mathcal{T}| =n)}
= k\mu_\varnothing(k)  \cdot \frac{B_{n}{\Pr{S_{n-1}=k}} }{(n-1) B_{n}{\rm BGW}^{\mu_\varnothing, \mu}( |\mathcal{T}| =n)}
\end{equation}
By (i) and the local limit theorem,  the convergence in \eqref{eq:convergence_degre_racine} therefore holds for every $k$ fixed. To obtain a uniform convergence, fix any $\varepsilon > 0$ and let $K \ge 1$ be such that $\sum_{j \ge K} j \mu_\varnothing(j) < \varepsilon$. Then
\begin{equation}
\sup_{1 \le k \le K} \big|k \mu_\varnothing(k) B_n \P(S_{n-1} = -k) - p_1(0) k \mu_\varnothing(k)\big| \cv 0,
\end{equation}
and, from \eqref{eq:thm_limite_locale},
\begin{equation}
\sup_{k \ge K} \big|k \mu_\varnothing(k) B_n \P(S_{n-1} = -k) - p_1(0) k \mu_\varnothing(k)\big|
\le \varepsilon (2\|p_1\| + o(1)),
\end{equation}
which establishes (ii).

We finally prove (iii). Let $\CRT_1, \dots, \CRT_k$ be $k$ independent {\GW} trees with offspring distribution $\mu$. To simplify notation, set $Z_{j}=\sum_{i=1}^j |\CRT_i|$ for $1 \leq j \leq k$. Fix $m \geq 0$. Note that, for $n>4m$,
$$\left\{\sup_{1 \le i \le k} |\CRT_i|=n-m, Z_{k}=n \right\} = \bigcup_{i=1}^{k} \left\{ |\CRT_i|=n-m, Z_{k}=n  \right\},$$
where the union is taken on disjoint events. As a consequence, by exchangeability of the vector $(|\CRT_1|, \dots, |\CRT_k|)$ under the conditional distribution $\P(\ \cdot \mid Z_{k} = n)$, we have
\begin{align*}
\P\bigg(\sup_{1 \le i \le k} |\CRT_i| = n-m \biggm| Z_{k} = n\bigg)
&= \sum_{i=1}^k \P(|\CRT_i| = n-m \mid Z_{k} = n)
\\
&= k \cdot \P(Z_1 = n-m \mid Z_{k} = n).
\end{align*}
Next, we have, for $n>4m$,
\begin{equation}
k \cdot \P(Z_{1} = n-m \mid Z_{k} = n)
= k \cdot \frac{\P(Z_{1} = n-m) \P(Z_{k-1} = m)}{\P(Z_{k} = n)}.
\end{equation}
Since $Z_{k}$ has the same law as the first hitting time of $-k$ by the random walk $S$, Kemperman’s formula yields
\begin{equation}
k \cdot \frac{\P(Z_{1} = n-m)}{\P(Z_{k} = n)}
= k \cdot\frac{\frac{1}{n-m} \P(S_{n-m} = -1)}{\frac{k}{n} \P(S_{n} = -k)}
\cv 1,
\end{equation}
where the convergence follows from the local limit theorem \eqref{eq:thm_limite_locale} and $B_{n-m}/B_{n} \to 1$. It follows that
$$\P\bigg(\sup_{1 \le i \le k} |\CRT_i| = n-m \biggm| Z_{k} = n\bigg)  \quad \mathop{\longrightarrow}_{n \rightarrow \infty} \quad \P(Z_{k-1} = m).$$
This completes the proof. 
\end{proof}

We finally prove Proposition \ref{prop:plus_grand_sous_arbre_GW_modifie} and Theorem \ref{thm:convergence_contour_arbre_SG}.

\begin{proof}[Proof of Proposition \ref{prop:plus_grand_sous_arbre_GW_modifie}] As in the proof of Lemma \ref{lem:maintool}, let $(\CRT_i)_{i \geq 1}$ be a sequence of independent {\GW} trees with offspring distribution $\mu$ and set $Z_{j}=\sum_{i=1}^j |\CRT_i|$ for every $j \geq 1$. Then observe that for every $i \ge 1$ fixed, under the conditional distribution ${\rm BGW}^{\mu_\varnothing, \mu}_n(\ \cdot \mid N_0=i)$, the $i$ subtrees of the root are distributed as a forest of $i$ independent {\GW} trees with the same offspring distribution $\mu$, conditioned to have total size $n-1$. Therefore, with the notation of Lemma \ref{lem:maintool}, for every $L \geq 1$ and $k \geq 0$,
\begin{eqnarray*}
{\rm BGW}^{\mu_\varnothing, \mu}_n(M =  n-1-k, N_{0} = L)
&=&  {\rm BGW}^{\mu_\varnothing, \mu}_n(M =  n-1-k\mid N_0=L) \cdot {\rm BGW}^{\mu_\varnothing, \mu}_n(N_0=L)
\\
&= & \P(M_{n-1}^{\mu,L} =  n-1-k) \cdot {\rm BGW}^{\mu_\varnothing, \mu}_n(N_0=L) \\
&   \displaystyle \mathop{\longrightarrow}_{n \rightarrow \infty} &  \P(Z_{L-1} = k) \cdot \Pr{ \mathcal{N}=L},
\end{eqnarray*}
where we have used Lemma \ref{lem:maintool} (ii) and (iii) for the last step. This completes the proof.
\end{proof}

\begin{rem}\label{rem:diff}In order to prove that {under ${\rm BGW}^{\mu_\varnothing, \mu}_n$, $M/n \to 1$ in probability as $n \to \infty$} when $\mu_\varnothing$ and $\mu$ are given by \eqref{eq:lois_de_reproduction_cas_uniforme}, Marckert \& Panholzer crucially use the fact that for every $k \geq 1$, conditionally given $N_{0}=k$, the laws ${\rm BGW}^{\mu_\varnothing, \mu}_n$ and ${\rm BGW}^{\mu}_n$ are the same. However, in the general case, $\mu_\varnothing$ and $\mu$ may have different supports. For this reason, we use an additional idea which consists in estimating the size of the largest tree in a forest of Bienaymé--Galton--Watson trees (Lemma \ref{lem:maintool} (iii)) and which also {allows us} to obtain a joint convergence in distribution in Proposition \ref{prop:plus_grand_sous_arbre_GW_modifie}.
\end{rem}

\begin{proof}[Proof of Theorem \ref{thm:convergence_contour_arbre_SG}]
We see from Proposition \ref{prop:plus_grand_sous_arbre_GW_modifie} that under ${\rm BGW}^{\mu_\varnothing, \mu}_n$, with probability tending to $1$ as $n \to \infty$,  the root has one subtree, say $\tau_{n}$, of size $M_{n} = n-o(n)$. Furthermore, conditional on $M_{n}$, this subtree is distributed as ${\rm BGW}^\mu_{M_{n}}$. We conclude from \eqref{eq:Duq}  that its associated rescaled {\L}ukasiewicz path $ ({B_{M_{n}}^{-1}} \W_{\lfloor M_{n} s \rfloor}(\tau_{n}),s \in [0,1])$
converges in distribution towards to $(\X(s) : s \in [0,1])$ as $n \rightarrow \infty$. Since all the other subtrees have total size $o(n)$ with high probability, their contribution does not affect the limit by standard properties of the Skorokhod topology, and the claim follows.
\end{proof}

\begin{rem}
As in \cite[Sec. 3.4]{MP02}, under the assumptions of Theorem \ref{thm:convergence_contour_arbre_SG}, we have in fact the joint convergence in distribution of the {rescaled} {\L}ukasiewicz path, the height process and the contour process of the trees to $(\X,\H,\H)$. Indeed, more than \eqref{eq:Duq}, Duquesne \cite{Du03} obtained this convergence for (non-modified) conditioned Bienaymé--Galton--Watson trees and the above argument extends verbatim. A consequence is for example that the height of the shape of $\T_{n}$ is of order $n/B_{n}$.
\end{rem}

\subsection{Application to degree-constrained noncrossing trees}

Our goal is now to prove Theorem \ref{thm:enumA}. Recall that $\NC_{n}^{\mathcal{A}}$ is the set of all noncrossing trees having $n$ vertices and with degrees only belonging to $ \mathcal{A} \subset \N$. Recall also from Sec.~\ref{sec:codings} the notation $\mathbb{C}(\tau)$ for a plane tree $\tau$ and the bijection $\Phi_{n}$ between $\NC_{n}$ and $\T_{n}^{\mathsf{dec}}$. We first introduce some notation. 
Denote by $\T_{n}^{ \mathcal{A} }$ the set of all plane trees having $n$ vertices and with degrees only belonging to $ \mathcal{A}$ and set $\T_{n}^{ \mathcal{A} ,\mathsf{dec}}= \{ (\tau, \mathbf{c}) \in \T_{n}^{\mathsf{dec}} : \tau \in \T_{n}^{ \mathcal{A} }\} $. It is clear that $\Phi_{n}$ also yields a bijection between $\NC_{n}^{\mathcal{A}}$ and $\T_{n}^{ \mathcal{A} ,\mathsf{dec}}$.

\begin{proof}[Proof of Theorem \ref{thm:enumA}] It is clear that $ \mathcal{A} \neq \{1\}$, otherwise  $\NC_{n}^{\mathcal{A}}  = \varnothing$ for every $n \geq 2$.  
We first construct a  uniform  element of $\NC_{n}^{\mathcal{A}}$ as follows. Set $w(k)= \mathbb{1}_{k \in \mathcal{A} }$. Recalling the definition of  $\Psi$ in \eqref{eq:Psi}, we have 
$$\Psi(1)=\frac{\sum_{k \in \mathcal{A}, k>1 } (k-1) k }{1+\sum_{k \in \mathcal{A},k>1} k}.$$ 
Then note that
$$\sum_{k \in \mathcal{A}, k>1 } (k-1) k-\sum_{k \in \mathcal{A},k>1} k= \sum_{k \in \mathcal{A}, k>1}k(k-2)>1,$$
since $ \mathcal{A} \neq \{1,2\}$. As a consequence, there exists $b \in (0,1)$ such that \eqref{eq:defb} holds, and we can consider   the probability measures $\mu^{ \mathcal{A} }$ and $\mu_{\varnothing}^{ \mathcal{A} }$ given by Theorem \ref{thm:arbre_SG_presque_GW}. More precisely, $$\mu ^{ \mathcal{A}}(k)= a (k+1) b^k \mathbbm{1}_{k+1 \in \mathcal{A} }, \qquad \mu^{ \mathcal{A} }_\varnothing(k)= cb^{k} \mathbbm{1}_{ k \in \mathcal{A} },$$
with $a=(\sum_{i+1 \in \mathcal{A}} (i+1) b^i)^{-1}$ and $c=(\sum_{i \in \mathcal{A}} b^{i})^{-1}$. Let $ \mathcal{T}_{n}$  be a ${\rm BGW}^{\mu_\varnothing^{\mathcal{A}}, \mu^{ \mathcal{A} }}_n$ tree and conditionally given $\mathcal{T}_{n}$,  let $\mathcal{C}(\mathcal{T}_{n})$ be a uniform element of $ \mathbb{C}(\mathcal{T}_{n})$. Finally, set $\mathscr{T}_{n}^{ \mathcal{A} }= \Phi_{n}^{-1} \left( ( \mathcal{T}_{n}, \mathcal{C}( \mathcal{T}_{n}) ) \right)$.
Then $ \mathscr{T}_{n}^{ \mathcal{A} }$  is uniformly distributed in $\NC_{n}^{\mathcal{A}}$. Indeed, this simply follows from the fact that $\Phi_{n}$ is a bijection between $\NC_{n}^{\mathcal{A}}$ and $\T_{n}^{ \mathcal{A} ,\mathsf{dec}}$ and that $\mathcal{T}_{n}$ is uniformly distributed on $\T_{n}^{ \mathcal{A} }$ by Theorem \ref{thm:arbre_SG_presque_GW}. 

Now fix $\tau \in \T_{n}^{ \mathcal{A} }$ and $\mathbf{c}\in \mathbb{C}(\tau)$. By the previous discussion, we have
$$ \frac{1}{ \# \NC_{n}^{\mathcal{A}}}= \Pr{( \mathcal{T}_{n}, \mathcal{C}( \mathcal{T}_{n}) )=(\tau,\mathbf{c})}= \Pr{ \mathcal{T}_{n}= \tau } \cdot \frac{1}{\# \mathbb{C}(\tau)}= \frac{{\rm BGW}^{\mu_\varnothing^{\mathcal{A}}, \mu^{ \mathcal{A} }}( \mathcal{T} =\tau)}{{\rm BGW}^{\mu_\varnothing^{\mathcal{A}}, \mu^{ \mathcal{A} }}(| \mathcal{T}|=n )} \cdot \frac{1}{\prod_{u \in \tau \setminus \{\varnothing\}  }(k_{u}+1)}.$$
However, by definition, 
$${\rm BGW}^{\mu_\varnothing^{\mathcal{A}}, \mu^{ \mathcal{A} }}( \mathcal{T} =\tau)=c b^{k_{\varnothing}} \cdot {\prod_{u \in \tau \setminus \{\varnothing\}} a (k_{u}+1) b^{k_{u}}} = c \cdot (ab)^{n-1} \cdot {\prod_{u \in \tau \setminus \{\varnothing\}}} (k_{u}+1).$$
As a consequence $\# \NC_{n}^{\mathcal{A}}= c^{-1}\cdot (ab)^{-(n-1)} \cdot {\rm BGW}^{\mu_\varnothing^{\mathcal{A}}, \mu^{ \mathcal{A} }}(| \mathcal{T}|=n )$.
Since  $\mu^{ \mathcal{A} }$ has finite variance, an adaptation of Lemma \ref{lem:maintool} (i) to the possibly periodic case yields
$${\rm BGW}^{\mu_\varnothing, \mu}( |\mathcal{T}| =n)  \quad \mathop{\sim}_{n \rightarrow \infty} \quad \gcd( \mathcal{A}-1) \cdot  \frac{1}{\sqrt{4\pi}}\cdot \left(\sum_{k \ge 1} k \mu_\varnothing(k) \right)  \cdot \frac{1}{n \cdot \sigma_{ \mathcal{A} } \sqrt{n/2}},$$
where $\sigma_{ \mathcal{A} }^{2}$ is the variance of $\mu^{ \mathcal{A} }$ and $n$ is chosen such that $n \equiv 2 \pmod{\gcd( \mathcal{A}-1)}$.  Hence
$$ \# \NC_{n}^{\mathcal{A}}  \quad \mathop{\sim}_{n \rightarrow \infty} \quad \gcd( \mathcal{A}-1)  \frac{1}{\sqrt{2\pi \sigma_{ \mathcal{A} }^{2}}} \cdot \left( \sum_{k \geq 1} k \mu^{ \mathcal{A} }_{\varnothing}(k) \right) \cdot \frac{1}{c}    \cdot (ab)^{-(n-1)} \cdot n^{-3/2}.$$
The conclusion follows.
\end{proof}

\section{Iterating laminations, \emph{ad libitum}}
\label{sec:extensions}

Recall that in Section \ref{sec:trilam}, we have constructed a triangulation $L(\X,\boldsymbol{\ell})$ from the stable lamination $L(\X)$ by triangulating each one of its faces. In the last part of this paper, we propose other ways to fill-in the faces of stable laminations. 

The study of {multiple} iterated real-valued processes has been triggered by the work of Curien \& Konstantopoulos \cite{CK14}, which were motivated by the iteration of two Brownian motions considered by Burdzy \cite{Bur93}. Casse \& Marckert \cite{CM15} then studied the iteration of reflected Brownian motion as well as the iteration of stable processes. Here we propose to iterate  laminations, in a sense that will be made precise in the following lines.

\begin{defn} Let $V$ be a face of a lamination of $\Db$. If $V$ is a triangle,  we say that $V$ is decorated by convention. Otherwise, a decoration of $V$ is an order preserving surjection $\phi_{V}: \S \rightarrow \partial V \cap \S$. Intuitively, we can view $\phi_{V}$ as an inverse of the evolution of the  ``number'' of vertices belonging to  $\partial V \cap \S$ as one goes around $\S$. 
A decorated lamination is by definition a lamination with a decoration associated with every face. 

Let $ (V, \phi_{V})$ be a decorated face and $ {L}$ be a  lamination of $\overline{\D}$. If $F$ is a face of $L$, set
$$V_{F}= \overline{ \bigcup_{[u,v] \in \partial F} [\phi_{V}(u),\phi_{V}(v)]}$$
and
$$ {V}(L)= \overline{V} \cup  \bigcup_{F \textrm{ face of } L} V_{F},$$
which is a lamination such that  every face of $ {V}(L)$ is the ``interior'' of $V_{F}$ for some face $F$ of $L$. In addition, if $ {L}$ is a \emph{decorated} lamination of $\overline{\D}$, $V(L)$ can be seen as a decorated lamination by setting $\phi_{V_{F}}= \phi_{V} \circ \phi_{F}$ for every decorated face $(F,\phi_{F})$ of $L$.

Now let $L^{0}$ be a decorated lamination, and let $ \mathcal{L}=(L_{V})_{V \textrm{ face of } L^{0}}$ be a collection of 
laminations indexed by the faces of $L^{0}$. Then set
$$  \mathcal{L} \circ L^{0} = \bigcup_{V \textrm{ face of } L^{0}} V(L_{V}).$$
It is possible to check that $\mathcal{L} \circ L^{0} $ is a lamination. Intuitively, it is obtained from $L^{0}$ by inserting the lamination $L_{V}$ inside each face $V$ of $L^{0}$.
In addition, if $ \mathcal{L}=(L_{V})_{V \textrm{ face of } L^{0}}$ is a collection of \emph{decorated} laminations, then $\mathcal{L} \circ L^{0} $ is a decorated lamination. \end{defn}

An important example is the $\alpha$-stable lamination $L(\X)$, which can be seen as a decorated lamination: if $\alpha \in (1,2)$ and if $u$ is a jump time of $\X$,  the bijection $p_{u}$ defined by \eqref{eq:pu} is a decoration of the face coded by $u$ (with the usual identification of $\S$ with $[0,1]$).  It is actually possible to check that given a stable lamination $\mathbf{L}_{\alpha}$, we can recover the decorations $p_{u}$ in a measurable way up to scaling factors by using approximations of local times, but we do not enter into the details since we do not require this fact.

\begin{defn}Fix $n \geq 1$ and let $\alpha_{1}, \ldots, \alpha_{n-1} \in (1,2)$ and $\alpha_{n} \in (1,2]$. Set $\boldsymbol{\alpha}=(\alpha_{1}, \ldots, \alpha_{n})$.  Then $ \mathbf{L}_{\boldsymbol{\alpha}}$ is the random decorated lamination defined recursively as follows. First, $ \mathbf{L}_{(\alpha_{1})}$ is just the $\alpha_{1}$-stable lamination (which is a decorated lamination as seen above). Next, conditionally given $\mathbf{L}_{(\alpha_{1}, \ldots, {\alpha}_{n-1})}$, let $ \mathcal{L}_{\alpha_{n}}=(\mathbf{L}^{F}_{\alpha_{n}})_{F \textrm{ face of } \mathbf{L}_{(\alpha_{1}, \ldots, {\alpha}_{n-1})}}$ be a collection of independent $\alpha_{n}$ stable laminations indexed by the faces of  $\mathbf{L}_{(\alpha_{1}, \ldots, {\alpha}_{n-1})}$, which we view as decorated as explained above. Then set
$$\mathbf{L}_{(\alpha_{1}, \ldots, {\alpha}_{n-1},\alpha_{n})}= \mathcal{L}_{\alpha_{n}} \circ  \mathbf{L}_{(\alpha_{1}, \ldots, {\alpha}_{n-1})}.$$
Intuitively, $\mathbf{L}_{(\alpha_{1}, \ldots, {\alpha}_{n-1},\alpha_{n})}$ is obtained from $ \mathbf{L}_{(\alpha_{1}, \ldots, {\alpha}_{n-1})}$ by inserting independent $\alpha_{n}$-stable laminations inside every face of $ \mathbf{L}_{(\alpha_{1}, \ldots, {\alpha}_{n-1})}$.
\end{defn}

Note that the lamination $\mathbf{L}_{(\alpha_{1}, \ldots, {\alpha}_{n})}$ is maximal if and only if $\alpha_{n}=2$. We believe that the Hausdorff dimension $\dim(\mathbf{L}_{(\alpha_{1}, \ldots, {\alpha}_{n})})$ is almost surely equal to
\begin{equation}
\label{eq:Hitere}\max \left( 2- \frac{1}{\alpha_{1}},1+ \frac{1}{\alpha_{1}} \left( 1- \frac{1}{\alpha_{2}}\right) , 1+ \frac{1}{\alpha_{1} \alpha _{2}} \left( 1- \frac{1}{\alpha_{3}} \right), \ldots, 1+ \frac{1}{\alpha_{1} \alpha_{2} \cdots \alpha_{n-1}} \left( 1- \frac{1}{\alpha_{n}} \right)  \right) .
\end{equation}
Indeed, the decorations of the faces of $\mathbf{L}_{(\alpha_{1}, \ldots, {\alpha}_{k})}$ are closely related to the iteration of stable subordinators of indices $1/\alpha_{1}, 1/\alpha_{2}, \ldots, 1/\alpha_{k}$, and one should be able to adapt \cite[Sec.~5]{Kor11} to show that the the boundaries of the faces of $\mathbf{L}_{(\alpha_{1}, \ldots, {\alpha}_{k})}$ restricted to $\S$ have Hausdorff dimension $(\alpha_{1} \cdots \alpha_{k})^{-1}$, so that  $\mathbf{L}_{(\alpha_{1}, \ldots, {\alpha}_{k})} \setminus \mathbf{L}_{(\alpha_{1}, \ldots, {\alpha}_{k-1})}$ has Hausdorff dimension $1+ \frac{1}{\alpha_{1} \alpha_{2} \cdots \alpha_{k-1}} \left( 1- \frac{1}{\alpha_{k}} \right) $. However, we have not worked out the details.

\begin{ques} If $\boldsymbol{\alpha} \neq \boldsymbol{\alpha'}$, is true that the laws of $ \mathbf{L}_{\boldsymbol{\alpha}}$ and $ \mathbf{L}_{\boldsymbol{\alpha'}}$ are singular with respect to each other?
\end{ques}
If $(\alpha_{1},\alpha_{2}) \neq (\alpha'_{1}, \alpha'_{2})$, assuming that \eqref{eq:Hitere} holds, one can check that $\dim(\mathbf{L}_{(\alpha_{1},\alpha_{2}) }) \neq \dim(\mathbf{L}_{(\alpha_{1},\alpha_{2}) })$. However, still assuming that \eqref{eq:Hitere} is true, we have $\dim(\mathbf{L}_{(1.1,1.2,2)})= \dim(\mathbf{L}_{(1.2,1.1,2)})$. Another direction would be to find out what happens to $\mathbf{L}_{(\alpha_{1}, \ldots, {\alpha}_{n})}$ as $n \rightarrow \infty$.

We believe that $\mathbf{L}_{(\alpha_{1}, \ldots, \alpha_{n})}$ is the scaling limit of a modified version of random dissections considered in \cite{Kor11}: instead of just choosing a random dissection of a large polygon according to critical Boltzmann weights in the domain of attraction of a stable law, first sample a random dissection with such Boltzmann weights in the domain of attraction of an $\alpha_{1}$-stable law, then inside each face of the dissection independently sample again a random dissection with Boltzmann weights in the domain of attraction of an $\alpha_{2}$-stable law, and so on. Similarly, as in \cite{KM15}, one can consider a random noncrossing partition with Boltzmann weights in the domain of attraction of an $\alpha_{1}$-stable law, then partition each block independently at random using a noncrossing partition with Boltzmann weights in the domain of attraction of an $\alpha_{2}$-stable law, and so on.

\begin{ques}In a certain sense, the $\alpha$-stable random lamination can be seen as the dual of the $\alpha$-stable tree. As was suggested to us by Nicolas Curien, iterating stable laminations can be alternatively seen as iterating stable trees. Roughly speaking, start with a stable tree of index $\alpha_{1}$, and then ``explode'' each branch point by gluing inside a stable tree of index $\alpha_{2}$, and so on. What is the Hausdorff dimension of the random tree constructed in this way? What happens  as $n \rightarrow \infty$? We hope to investigate this in a future work. 
\end{ques}

Note that if one starts with a stable tree and explodes each branchpoint by simply gluing inside a ``loop'', one gets the so-called stable looptrees which were introduced and studied in \cite{CK13}. More generally, one can imagine exploding branchpoints in stable trees and glue inside any compact metric space equipped with a homeomorphism with $[0,1]$.

\bibliographystyle{siam}
%\bibliography{bibli}

\end{document}